\documentclass[12pt]{amsart}
\usepackage{hyperref}
\usepackage[all]{xy}
\usepackage{amsmath,amssymb,amsthm,mathrsfs,pifont}
\usepackage{txfonts}
\usepackage{enumerate}
\usepackage{graphicx}
\usepackage{subcaption}
\usepackage[figuresright]{rotating}
\usepackage[margin=24mm]{geometry}
\theoremstyle{plain}
	\newtheorem{thm}{Theorem}[section]
	\newtheorem{prp}[thm]{Proposition}
	\newtheorem{lem}[thm]{Lemma}

\theoremstyle{definition}
	\newtheorem{dfn}[thm]{Definition}
	\newtheorem{ex}[thm]{Example}
	\newtheorem{problem}[thm]{Problem}
	
\theoremstyle{remark}
	\newtheorem{rem}[thm]{Remark}

\newcommand{ \map}{\operatorname{map}}
\newcommand{ \Map}{\operatorname{Map}}

\newcommand{ \cat}{\operatorname{cat}}

\newcommand{ \Prin}{\operatorname{Prin}}
\newcommand{ \eq}{\operatorname{eq}}

\newcommand{ \conj}{\operatorname{conj}}

\newcommand{ \op}{\operatorname{op}}

\newcommand{ \id}{\operatorname{id}}

\newcommand{ \Aut}{\operatorname{Aut}}
\newcommand{ \SO}{\operatorname{SO}}
\newcommand{ \SU}{\operatorname{SU}}
\newcommand{ \Sp}{\operatorname{Sp}}

\newcommand{\cmark}{\ding{51}}
\newcommand{\xmark}{\ding{55}}

\begin{document}
\title{Higher homotopy normalities in topological groups}
\author{Mitsunobu Tsutaya}
\address{Faculty of Mathematics, Kyushu University, Fukuoka 819-0395, Japan}
\email{tsutaya@math.kyushu-u.ac.jp}
\date{}
\subjclass[2010]{55P45 (primary), 55R70 (secondary)}
\thanks{The author was supported by JSPS KAKENHI 19K14535.}
\begin{abstract}
The purpose of this paper is to introduce $N_k(\ell)$-maps ($1\le k,\ell\le\infty$), which describe higher homotopy normalities, and to study their basic properties and examples.
An $N_k(\ell)$-map is defined with higher homotopical conditions.
It is shown that a homomorphism is an $N_k(\ell)$-map if and only if there exists fiberwise maps between fiberwise projective spaces with some properties.
Also, the homotopy quotient of an $N_k(k)$-map is shown to be an $H$-space if its LS category is not greater than $k$.
As an application, we investigate when the inclusions $\SU(m)\to\SU(n)$ and $\SO(2m+1)\to\SO(2n+1)$ are $p$-locally $N_k(\ell)$-maps.
\end{abstract}
\maketitle

\section{Introduction}
\label{section_intro}

A \textit{normal subgroup} $H\subset G$ of a topological group $G$ is defined to be a subgroup closed under the conjugation by $G$.
A crossed module is a natural generalization of a normal subgroup, which plays crucial roles in homotopy theory.
\begin{dfn}
\label{dfn_crossed}
A \textit{(topological) crossed module} is topological groups $H$ and $G$ equipped with continuous group homomorphisms
\[
f\colon H\to G
\quad
\text{and}
\quad
\rho\colon G\to\Aut(H)
\]
satisfying the following conditions:
\begin{enumerate}
\item
$\rho(f(h))(x)=hxh^{-1}$ for any $h,x\in H$,
\item
$f(\rho(g)(x))=gf(x)g^{-1}$ for any $x\in H$ and $g\in G$.
\end{enumerate}
\end{dfn}
In this paper, we propose and investigate higher homotopy variants of crossed module, which sit between homomorphisms (without any special property) and crossed modules.

Ordinary and higher homotopy normalities have been extensively studied.
McCarty \cite{MR176471} defined a homotopy normality as follows: a subgroup $H\subset G$ is homotopy normal if the conjugation $G\times H\to G$, $(g,h)\mapsto ghg^{-1}$ is homotopic to a map into $H$ through a homotopy of maps between topological pairs $(G\times H,H\times H)\to(G,H)$.
James \cite{MR233376} defined a weaker homotopy normality as follows: a subgroup $H\subset G$ is homotopy normal if the conjugation $G\times H\to G$ is homotopic to a map into $H$ through an ordinary homotopy of continuous maps.
These conditions are recognized to be ``the first order'' homotopy normality in our viewpoint.
But these homotopy normalities do not guarantee the existence of a multiplicative structure on the quotient $G/H$.

There have been several works on investigating these kinds of homotopy normalities of Lie groups.
The methods adopted so far have been applications of the (relative) Samelson products \cite{MR176471,MR233376,MR303532,MR0431239,MR660053,MR780349,MR891614,MR3733837} and the Hopf algebra structures on (generalized) cohomology groups \cite{MR1322603,MR1669987,MR1825825,MR2028671,MR2320357}.

For higher homotopy normality, it has been focused on where a given homomorphism is placed in a homotopy fiber sequence \cite{MR2563286,MR2928911}.
In general, a homomorphism $f\colon H\to G$ induces the homotopy fiber sequence
\begin{align}
\label{eq_fibseq_H->G}
\cdots
\to
H
\xrightarrow{f}
G
\to
EH\times_fG
\to
BH
\xrightarrow{Bf}
BG,
\end{align}
where $EH\times_fG$ is the Borel construction by the action of $H$ on $G$ through $f$.
If we suppose $H\subset G$ is a normal subgroup, then the above sequence can be extended as follows:
\[
\cdots
\to
H
\to
G
\to
G/H
\to
BH
\to
BG
\to
B(G/H).
\]
Actually, this is obtained by applying the construction of the homotopy fiber sequence (\ref{eq_fibseq_H->G}) to the homomorphism $G\to G/H$.
This construction is generalized to crossed modules by Farjoun and Segev \cite{MR2563286}.
Although their construction is for simplicial groups, it also works for topological groups.
Conversely, they called $f\colon H\to G$ a homotopy normal map if the homotopy fiber sequence (\ref{eq_fibseq_H->G}) can be extended as follows by some map $BG\to W$:
\begin{align*}
\cdots
\to
H
\xrightarrow{f}
G
\to
EH\times_fG
\to
BH
\xrightarrow{Bf}
BG
\to
W.
\end{align*}
In particular, the homomorphism $H\to\ast$ to the trivial group is homotopy normal in the sense of Farjoun and Segev if and only if $BH$ is a loop space.

The main objective of this paper is to introduce another class of higher homotopy normality called \textit{$N_k(\ell)$-maps} for $1\le k,\ell\le\infty$ (Definition \ref{dfn_(k,l)-normal}).
In particular, we see that a homomorphism is an $N_1(1)$-map if and only if it is homotopy normal in the sense of McCarty (Remark \ref{rem_James}).
Also, we show that a homomorphism $H\to\ast$ to the trivial group is an $N_k(\ell)$-map if and only if $H$ is a $C(k,\ell)$-space (Theorem \ref{thm_Nkl_to_trivial}), which is introduced in \cite{MR2678992}.
As a consequence, $H$ is a $C(\infty,\infty)$-space if and only if $BH$ is an $H$-space.
This implies that our homotopy normality is much weaker than that of Farjoun and Segev.

The advantage of our homotopy normality is a good connection with fiberwise homotopy theory.
Actually, our main theorem is roughly stated as follows:
a homomorphism $H\to G$ is an $N_k(\ell)$-map if and only if the induced fiberwise homomorphism of group bundles $E_kH\times_{\conj}H\to E_kG\times_{\conj}G$ associated to the conjugation action of a group on itself admits a factorization as a fiberwise $A_\ell$-map 
\[
E_kH\times_{\conj}H\to E\to E_kG\times_{\conj}G
\]
through some fiberwise $A_\ell$-space $E$ over $B_kG$ with fiber equivalent to $H$ (Theorem \ref{thm_main}).
The latter condition can be checked by the obstruction theory of fiberwise projective spaces, which provides us a method to study $N_k(\ell)$-maps (Sections \ref{section_projective} and \ref{section_Lie}).

We also discuss when the homotopy quotient $X=EH\times_fG$ of a homomorphism $f\colon H\to G$ is an $H$-space.
We show that $X$ is an $H$-space if $f$ is an $N_k(k)$-map and the LS-category $\cat X$ of $X$ is estimated as $\cat X\le k$ (Theorem \ref{thm_quotient_Hstr}). 

Our argument is based on the category of topological monoids and $A_n$-maps between them constructed in \cite{MR3491849}.
This is just a choice of a model among possible higher categorical settings for $A_n$-maps between $A_\infty$-spaces.
Our definitions and results in the present work should be valid in other settings.

This paper is constructed as follows.
In Section \ref{section_monoid}, we recall the category of topological monoids and $A_n$-maps between them introduced in \cite{MR3491849}.
In Sections \ref{section_associahedra} and \ref{section_an-space}, we reformulate $A_n$-space and $A_n$-map from an $A_n$-space to a topological monoid.
Our composition of $A_n$-maps between topological monoids with an $A_n$-map from an $A_n$-space to a topological monoid is defined to be associative.
In Section \ref{section_fwAn}, we recall the classification theorem of fiberwise $A_n$-spaces established in \cite{MR2876314,MR3327166}, which plays a central role in the proof of the main theorem.
In Section \ref{section_normal}, we introduce $N_k(\ell)$-maps and prove our main result Theorem \ref{thm_main}.
In Section \ref{section_C(k,l)}, we discuss the relation between $C(k,\ell)$-space and $N_k(\ell)$-map.
In Section \ref{section_quotient}, we discuss the $H$-structure on homotopy quotient.
In Section \ref{section_projective}, we recall the basic properties of fiberwise projective spaces and give some criterion for non-$N_k(\ell)$-maps (Theorem \ref{thm_Nkl_fwproj}).
In Section \ref{section_Lie}, we investigate when the inclusions $\SU(m)\to\SU(n)$ and $\SO(2m+1)\to\SO(2n+1)$ are $p$-locally $N_k(\ell)$-maps.

In this paper, we always work in the category of compactly generated spaces $\mathbf{CG}$.
So the adjunction between products and mapping spaces is always available.


\section{Category of topological monoids and $A_n$-maps}
\label{section_monoid}

We begin by recalling the construction of the topologically enriched category of topological monoids and $A_n$-maps introduced in \cite[Section 4]{MR3491849}.
Proofs will be omitted here.
Let $[0,\infty]$ be the one point compactification of $[0,\infty)$, which is homeomorphic to the unit interval $[0,1]$.

\begin{dfn}
Let $f\colon G\to G'$ be a pointed map between topological monoids.
A pair $(\{f_i\}_i,L)$ of a family of maps $\{f_i\colon[0,\infty]^{i-1}\times G^i\to G'\}_{i=1}^n$ and $L\in[0,\infty]$ is called an \textit{$A_n$-form} on $f$ if the following conditions hold:
\begin{enumerate}
\item
for any $x\in G$, $f_1(x)=f(x)$,
\item
for any $2\le i\le n$, $1\le k\le i-1$, $t_1,\ldots,t_{i-1}\in[0,\infty]$ and $x_1,\ldots,x_i\in G$,
\begin{align*}
&f_i(t_1,\ldots,t_{i-1};x_1,\ldots,x_i)\\
&=
\begin{cases}
f_{i-1}(t_1,\ldots,t_{k-1},t_{k+1},\ldots,t_{i-1};x_1,\ldots,x_{k-1},x_kx_{k+1},x_{k+2},\ldots,x_i)
	& \text{if $t_k=0$},\\
f_k(t_1,\ldots,t_{k-1};x_1,\ldots,x_k)f_{i-k}(t_{k+1},\ldots,t_{i-1};x_{k+1},\ldots,x_i)
	& \text{if $t_k\ge L$},
\end{cases}
\end{align*}
\item
for any $2\le i\le n$, $1\le k\le i$, $t_1,\ldots,t_{i-1}\in[0,\infty]$ and $x_1,\ldots,x_{i-1}\in G$,
\begin{align*}
&f_i(t_1,\ldots,t_{i-1};x_1,\ldots,x_{k-1},\ast,x_k,\ldots,x_{i-1})\\
&=
\begin{cases}
f_{i-1}(t_2,\ldots,t_{i-1};x_1,\ldots,x_{i-1})
	& \text{if $k=1$},\\
f_{i-1}(t_1,\ldots,t_{k-2},\max\{t_{k-1},t_k\},t_{k+1},\ldots,t_{i-1};x_1,\ldots,x_{k-1},x_k,\ldots,x_{i-1})
	& \text{if $1<k<i$},\\
f_{i-1}(t_1,\ldots,t_{i-2};x_1,\ldots,x_{i-1})
	& \text{if $k=i$}.
\end{cases}
\end{align*}
\end{enumerate}
A pair $(f,(\{f_i\}_i,L))$ is called an \textit{$A_n$-map}.
The space of $A_n$-maps between topological monoids $G$ and $G'$ is denoted by $\mathcal{A}_n(G,G')$.
\end{dfn}
The composition is given as follows.

\begin{dfn}
The composition $(h,\{h_i\}_i,L+L')=g\circ f$ of $f=(f,\{f_i\}_i,L)\in\mathcal{A}_n(G,G')$ and $g=(g,\{g_i\}_i,L')\in\mathcal{A}_n(G',G'')$ defined as follows:
for any $r,i_1,\ldots,i_r\ge1$ with $i_1+\cdots+i_r\le n$, $t_1,\ldots,t_{r-1}\in[0,\infty]$, $\mathbf{s}_k\in[0,\infty]^{i_k-1}$, $\mathbf{x}_k\in G^{i_k}$, we have
\begin{align*}
h_{i_1+\cdots+i_r}(\mathbf{s}_1,t_1+L,\mathbf{s}_2,t_2+L,\ldots,t_{r-1}+L,\mathbf{s}_r;
	\mathbf{x}_1,\ldots,\mathbf{x}_k)
=
g_r(t_1,\ldots,t_{r-1};f_{i_1}(\mathbf{s}_1;\mathbf{x}_1),\ldots,f_{i_r}(\mathbf{s}_r;\mathbf{x}_r)).
\end{align*}
\end{dfn}

A homomorphism $f\colon G\to H$ can be seen as the $A_n$-map equipped with the \textit{standard $A_n$-form} $(\{f_i\}_i,0)$ given as
\[
f_i(t_1,\ldots,t_{i-1};g_1,\ldots,g_i)=f(g_1\cdots g_i)=f(g_1)\cdots f(g_i).
\]
In particular, the identity map $\operatorname{id}_G\in\mathcal{A}_n(G,G)$ acts as an identity for the above composition.

The following is proved in \cite[Section 4]{MR3491849}.

\begin{thm}
Topological monoids and the spaces of $A_n$-maps between them form a topologically enriched category $\mathcal{A}_n$.
Moreover, the category of topological monoids and homomorphisms can be embedded in $\mathcal{A}_n$ with the standard $A_n$-forms.
\end{thm}

The categories of left and right actions of topological monoids on topological spaces and $A_n$-equivariant maps are similarly defined.

\begin{dfn}
Let $G$ and $G'$ be topological monoids acting from the left on $X$ and $X'$, respectively, and $f=(f,(\{f_i\}_i,L))\colon G\to G'$ be an $A_n$-map between topological monoids.
A family of maps $\{\phi_i\colon[0,\infty]^i\times G^i\times X\to X'\}_{i=0}^n$ is an \textit{$A_n$-form} on a map $\phi\colon X\to X'$ if the following conditions hold: 
\begin{enumerate}
\item
for any $x\in X$, $\phi_0(x)=\phi(x)$,
\item
for any $1\le i\le n$, $1\le k\le i$, $t_1,\ldots,t_i\in[0,\infty]$, $g_1,\ldots,g_i\in G$ and $x\in X$,
\begin{align*}
&\phi_i(t_1,\ldots,t_i;g_1,\ldots,g_i;x)\\
&=
\begin{cases}
\phi_{i-1}(t_1,\ldots,t_{k-1},t_{k+1},\ldots,t_i;g_1,\ldots,g_{k-1},g_kg_{k+1},g_{k+2},\ldots,g_i;x)
	& \text{if $t_k=0$ and $k<i$},\\
\phi_{i-1}(t_1,\ldots,t_{i-1};g_1,\ldots,g_{i-1},g_ix)
	& \text{if $t_i=0$},\\
f_k(t_1,\ldots,t_{k-1};g_1,\ldots,g_k)\phi_{i-k}(t_{k+1},\ldots,t_i;g_{k+1},\ldots,g_i;x)
	& \text{if $t_k\ge L$},
\end{cases}
\end{align*}
\item
for any $1\le i\le n$, $1\le k\le i$, $t_1,\ldots,t_i\in[0,\infty]$, $g_1,\ldots,g_{i-1}\in G$ and $x\in X$,
\begin{align*}
&\phi_i(t_1,\ldots,t_i;g_1,\ldots,g_{k-1},\ast,g_k,\ldots,g_{i-1};x)\\
&=
\begin{cases}
\phi_{i-1}(t_2,\ldots,t_i;g_1,\ldots,g_{i-1};x)
	& \text{if $k=1$},\\
\phi_{i-1}(t_1,\ldots,t_{k-2},\max\{t_{k-1},t_k\},t_{k+1},\ldots,t_i;
	g_1,\ldots,g_{k-1},g_k,\ldots,g_{i-1};x)
	& \text{if $1<k\le i$}.
\end{cases}
\end{align*}
\end{enumerate}
A triple $(f,(\{f_i\}_i,L),\{\phi_i\}_i)$ is called an \textit{$A_n$-equivariant map}.
The space of $A_n$-equivariant maps between a left $G$-space $X$ and a left $G'$-space $X'$ is denoted by $\mathcal{A}^{\mathrm{L}}_n((G,X),(G',X'))$.
Moreover, the topologically enriched category $\mathcal{A}^{\mathrm{L}}_n$ of spaces with left actions of topological monoids and $A_n$-equivariant maps is similarly defined.
Equivariant $A_n$-maps between spaces with \textit{right} actions of topological monoids are similarly defined.
The category and a mapping space in it are denoted by $\mathcal{A}^{\mathrm{R}}_n$ and $\mathcal{A}^{\mathrm{R}}_n((X,G),(X',G'))$, respectively.
\end{dfn}

Our bar construction functor is defined on the category of $A_n$-equivariant maps.
We take a model of the $i$-dimensional simplex $\Delta^i$ as
\[
\Delta^i=
\{(t_0,\ldots,t_i)\in[0,\infty]^i\mid\text{$t_k=\infty$ for some $k$}\}.
\]
The face $\partial_k\colon\Delta^{i-1}\to\Delta^i$ and degeneracy $\epsilon_k\colon\Delta^{i+1}\to\Delta^i$ ($k=0,\ldots,i$) are given by
\begin{align*}
\partial_k(t_0,\ldots,t_{i-1})&=(t_0,\ldots,t_{k-1},0,t_k,\ldots,t_{i-1}),\\
\epsilon_k(t_0,\ldots,t_{i+1})
&=(t_0,\ldots,t_{k-1},\max\{t_k,t_{k+1}\},t_{k+2},\ldots,t_{i+1}).
\end{align*}

Consider the fiber product category $\mathcal{A}_n^{\mathrm{R}}\times_{\mathcal{A}_n}\mathcal{A}_n^{\mathrm{L}}$ in the obvious sense, where an object is a triple $(X,G,Y)$ of a topological monoid $G$, a right $G$-space $X$ and a left $G$-space $Y$ and a morphism is an $A_n$-equivariant map between them.

\begin{dfn}
For a triple $(X,G,Y)\in\mathcal{A}_n^{\mathrm{R}}\times_{\mathcal{A}_n}\mathcal{A}_n^{\mathrm{L}}$, the space $B_n(X,G,Y)$ is defined to be the quotient space
\[
B_n(X,G,Y)
=
\left(\coprod_{0\le i\le n}\Delta^i\times X\times G^i\times Y\middle)\right/{\sim}
\]
by the usual simplicial relation.
For a morphism
\[
(\phi,f,\psi)=(\phi,\{\phi_i\}_i,f,\{f_i\}_i,\psi,\{\psi_i\}_i,L)\colon
(X,G,Y)\to(X',G',Y')
\]
in $\mathcal{A}_n^{\mathrm{R}}\times_{\mathcal{A}_n}\mathcal{A}_n^{\mathrm{L}}$, the induced map
\[
B_n(\phi,f,\psi)\colon B_n(X,G,Y)\to B_n(X',G',Y')
\]
is defined by
\begin{align*}
&B_n(\phi,f,\psi)[\mathbf{s}_1,t_1+L,\mathbf{s}_2,t_2+L,\ldots,t_{r-1}+L,\mathbf{s}_r;
	x,\mathbf{g}_1,\ldots,\mathbf{g}_r,y]\\
&=[t_1,\ldots,t_{r-1};\phi_{i_1}(\mathbf{s}_1;x,\mathbf{g}_1),f_{i_2}(\mathbf{s}_2;\mathbf{g}_2),
	\ldots,f_{i_{r-1}}(\mathbf{s}_{r-1};\mathbf{g}_{r-1}),\psi_{i_r}(\mathbf{s}_r;\mathbf{g}_r,y)].
\end{align*}
for $i=i_1+\cdots+i_r\le n$, $t_1,\ldots,t_{r-1}\in[0,\infty]$, $\mathbf{s}_k\in[0,L]^{i_k-1}$, $x\in X$, $\mathbf{g}_k\in G^{i_k}$, $y\in Y$.
This construction defines a continuous functor
\[
B_n\colon
\mathcal{A}_n^{\mathrm{R}}\times_{\mathcal{A}_n}\mathcal{A}_n^{\mathrm{L}}
\to
\mathbf{CG}.
\]
In particular, the correspondence $G\mapsto B_nG=B_n(\ast,G,\ast)$ defines the \textit{$n$-th projective space} functor
\[
B_n\colon\mathcal{A}_n\to\mathbf{CG}_\ast.
\]
For $(X,G,Y)\in\mathcal{A}_\infty^{\mathrm{R}}\times_{\mathcal{A}_\infty}\mathcal{A}_\infty^{\mathrm{L}}$, let
\[
\iota_n\colon B_n(X,G,Y)\to B(X,G,Y)=B_\infty(X,G,Y)
\]
denote the natural inclusion.
\end{dfn}

Note that our bar construction functor coincides with the bar construction for usual equivariant maps through the obvious embedding into the category $\mathcal{A}_n^{\mathrm{R}}\times_{\mathcal{A}_n}\mathcal{A}_n^{\mathrm{L}}$.

The following is a technical lemma found in \cite[Theorem 7.6]{MR370579}.

\begin{lem}
\label{lem_bar_quasifib}
Let $G$ be a grouplike topological monoid with basepoint having the homotopy extension property and $(X,G,Y)\in\mathcal{A}_n^{\mathrm{R}}\times_{\mathcal{A}_n}\mathcal{A}_n^{\mathrm{L}}$.
Then the maps
\[
B_n(X,G,Y)\to B_n(X,G,\ast)
\quad
\text{and}
\quad
B_n(X,G,Y)\to B_n(\ast,G,Y)
\]
are quasifibrations.
\end{lem}

The following is the main theorem in \cite{MR3491849}.

\begin{thm}
\label{thm_Tsu16}
Let $G$ be a topological monoid, of which the underlying space is a CW complex, and $G'$ be a grouplike topological monoid, of which the basepoint has the homotopy extension property.
Then the following composite is a weak homotopy equivalence:
\[
\mathcal{A}_n(G,G')
\xrightarrow{B_n}
\Map_\ast(B_nG,B_nG')
\xrightarrow{(\iota_n)_{\sharp}}
\Map_\ast(B_nG,BG').
\]
\end{thm}

This homotopy equivalence establishes the one-to-one correspondence between the homotopy classes of $A_n$-maps $G\to G'$ and the basepoint-preserving maps $B_nG\to BG'$.
The homotopy classes of an $A_n$-map and the corresponding basepoint-preserving map are said to be \textit{adjoint} to each other.
The reason for this is the homotopy equivalence induces the adjunction
\[
\pi_0(\mathcal{A}_n(G,\Omega^{\mathrm{M}} X))\cong\pi_0(\Map_\ast(B_nG,X))
\]
between the projective space functor $B_n$ and the Moore based loop space functor $\Omega^{\mathrm{M}}$ in the homotopy categories.

\section{Planar rooted trees and associahedra}
\label{section_associahedra}

Let us recall the construction of associahedra in \cite{MR0420609}.
But our construction is slightly different from it since we need to make the composition of $A_n$-maps associative.

We set some notions related to planar rooted trees.
\begin{itemize}
\item
A \textit{rooted tree} in our sense is a contractible finite graph with distinguished vertex called the \textit{root} such that the root is of degree $1$ and no vertex is of degree $2$.
\item
A \textit{planar rooted tree} is an equivalence class of embeddings of a rooted tree into the upper half plane $\{(x,y)\in\mathbb{R}^2\mid y\ge0\}$ such that the root is mapped to the origin.
Two such embeddings of rooted tree are said to be \textit{equivalent} if they are isotopic through an isomorphism between rooted trees.
\item
In a planar rooted tree, a vertex of degree $1$ different from the root is called a \textit{leaf}.
We will write $\mathcal{T}_n$ for the set of planar rooted trees with $n$ leaves.
For $\tau\in\mathcal{T}_n$, we assign numbers $1,\ldots,n$ to the leaves of $\tau$ from left to right and call the leaf corresponding to $k$ the \textit{$k$-th leaf}.
\item
We will call the edge connected to the root the \textit{root edge}, an edge connected to a leaf a \textit{leaf edge} and other edges \textit{internal edges}.
We will write $I(\tau)$ for the set of internal edges of a planar rooted tree $\tau$.
\item
Along the shortest paths from leaves to the root, an orientation is assigned to each edge in a planar rooted tree.
\end{itemize}

Consider the space
\[
\mathcal{LT}_n=\coprod_{\tau\in\mathcal{T}_n}\{\tau\}\times[0,\infty]^{I(\tau)},
\]
where $[0,\infty]^{I(\tau)}$ is the set of maps $I(\tau)\to[0,\infty]$.
An \textit{elementary collapse} of a planar rooted tree $\tau$ at $e\in I(\tau)$ is a planar rooted tree obtained by just collapsing $e$ to a point.
Define the equivalence relation on $\mathcal{LT}_n$ generated by $(\tau,\ell)\sim(\tau',\ell')$ such that $\tau'$ is an elementary collapse of $\tau$ at $e\in I(\tau)$, $\ell(e)=0$, and, considering $I(\tau')\subset I(\tau)$, the function $\ell$ restricts to $\ell'$ on $I(\tau')$.
Denote by $\mathcal{K}_n$ the quotient space of $\mathcal{LT}_n$, called the \textit{$n$-th associahedron}.
In the rest of the paper, an element of $\mathcal{LT}_n$ or $\mathcal{K}_n$ will be denoted simply by $\tau$.
The function $I(\tau)\to[0,\infty]$ will be denoted by $\ell$ for any $\tau$.
Such a convention does not cause a confusion.
\begin{figure}[h!]
	\includegraphics[height=24mm]{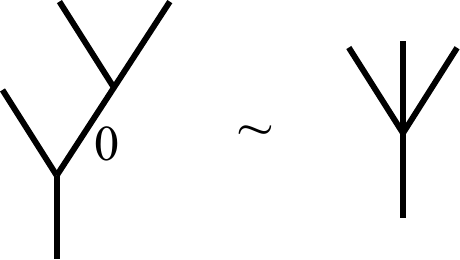}
	\caption{The equivalence by an elementary collapse.}
\end{figure}

Let $\rho\in\mathcal{T}_r$, $\sigma\in\mathcal{T}_s$ and $1\le k\le r$ ($k\in\mathbb{Z}$).
Then the \textit{grafting} $\partial_k(\rho,\sigma)\in\mathcal{T}_{r+s-1}$ of $\sigma$ to $\rho$ at the $k$-th leaf is obtained by identifying the root edge of $\sigma$ and the $k$-th leaf edge respecting the orientation.
Here we note:
\[
I(\partial_k(\rho,\sigma))=I(\rho)\sqcup I(\sigma)\sqcup\{\text{the new internal edge}\},
\]
where the new internal edge is the identified edge in the construction.
We can extend the grafting construction to
\[
\partial_k^L\colon\mathcal{LT}_r\times\mathcal{LT}_s\to\mathcal{LT}_{r+s-1}
\quad
\text{and}
\quad
\partial_k^L\colon\mathcal{K}_r\times\mathcal{K}_s\to\mathcal{K}_{r+s-1}
\]
for $L\in[0,\infty]$ by defining $\ell\colon I(\partial_k^L(\rho,\sigma))\to[0,\infty]$ as
\[
\ell(e)=
\begin{cases}
\ell(e) & \text{if $e\in I(\rho)\sqcup I(\sigma)$},\\
L & \text{if $e$ is the new internal edge}.
\end{cases}
\]
When $L=\infty$, we will simply write $\partial_k=\partial_k^\infty$.

Let $\tau\in\mathcal{T}_n$ ($n\ge3$) and $1\le k\le n$.
Let us construct the \textit{degeneracy} $s_k(\tau)\in\mathcal{T}_{n-1}$ as follows.
Let $s_k(\tau)$ be a planar rooted tree obtained from just removing the $k$-th leaf edge of $\tau$ if the resulting tree does not have a vertex of degree $2$ (i.e. the $k$-th leaf edge is not connected to a vertex of degree $3$).
When it has a vertex $v$ of degree $2$, define a planar rooted tree $s_k(\tau)$ identifying its edges connected to $v$ respecting the orientation.
To extend the degeneracy to the maps
\[
s_k\colon\mathcal{LT}_n\to\mathcal{LT}_{n-1}
\quad
\text{and}
\quad
s_k\colon\mathcal{K}_n\to\mathcal{K}_{n-1},
\]
we define $\ell(e_0)=\max\{\ell(e_0'),\ell(e_0'')\}$ for the edge $e_0\in I(s_k(\tau))$ obtained by identifying $e_0',e_0''\in I(\tau)$ if it exists and is internal.
The same value $\ell(e)$ as in $\tau$ is assigned for any other internal edge $e\in I(s_k(\tau))$.
\begin{figure}[h!]
	\includegraphics[height=24mm]{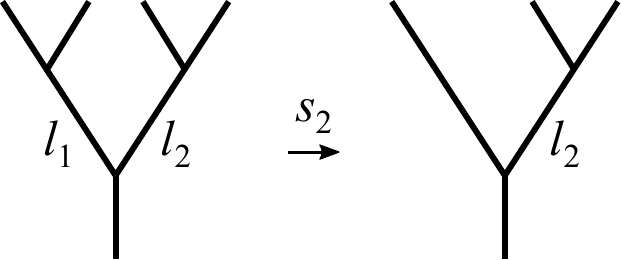}
	\caption{Degeneracy $s_2$ on $\mathcal{LT}_4$.}
\end{figure}

\section{$A_n$-spaces and $A_n$-maps}
\label{section_an-space}

In this section, we recall the definitions of $A_n$-spaces and of $A_n$-maps from $A_n$-spaces to topological monoids.

\begin{dfn}
Let $H$ be a pointed space.
A family of maps $\{m_i\colon\mathcal{K}_i\times H^i\to H\}_{i=2}^n$ is called an \textit{$A_n$-form} on $H$ if the following conditions hold:
\begin{enumerate}
\item
for any $r,s\ge2$ with $r+s-1\le n$, $1\le k\le r$, $\rho\in\mathcal{K}_r$, $\sigma\in\mathcal{K}_s$ and $x_1,\ldots,x_{r+s-1}\in H$,
\begin{align*}
m_{r+s-1}(\partial_k(\rho,\sigma);x_1,\ldots,x_{r+s-1})
=
m_r(\rho;x_1,\ldots,x_{k-1},m_s(\sigma;x_k,\ldots,x_{k+s-1}),x_{k+s},\ldots,x_{r+s-1}),
\end{align*}
\item
for any $x\in H$,
\[
m_2(\ast;x,\ast)=m_2(\ast;\ast,x)=x,
\]
\item
for any $3\le i\le n$, $1\le k\le i$, $\tau\in\mathcal{K}_i$ and $x_1,\ldots,x_{i-1}\in H$,
\begin{align*}
m_i(\tau;x_1,\ldots,x_{k-1},\ast,x_k,\ldots,x_{i-1})
=
m_{i-1}(s_k(\tau);x_1,\ldots,x_{k-1},x_k,\ldots,x_{i-1}).
\end{align*}
\end{enumerate}
A pair $(H,\{m_i\}_i)$ of a pointed space and an $A_n$-form on it is called an \textit{$A_n$-space}.
\end{dfn}

Similarly, we can define an $A_n$-form $\{\mu_i\colon\mathcal{K}_{i+1}\times H^{i}\times X\to X\}_{i=1}^n$ of $A_n$-actions on $X$ by $H$.
Taking the adjoint, this can be considered as an ``$A_n$-map'' from $H$ to $\Map(X,X)$.
This idea is extended to define an $A_n$-map from an $A_n$-map to a topological monoid as follows while our definition includes the parameter $L$ not appearing in the one by Stasheff \cite{MR0270372}.

\begin{dfn}
Let $(H,\{m_i\}_i)$ be an $A_n$-space, $G$ a topological monoid and $f\colon H\to G$ a pointed map.
A pair $(\{f_i\}_i,L)$ of a family of maps $\{f_i\colon\mathcal{K}_{i+1}\times H^i\to G\}_{i=1}^n$ and $L\in[0,\infty]$ is called an \textit{$A_n$-form} on $f$ if the following conditions hold:
\begin{enumerate}
\item
for any $r\ge1$, $s\ge2$ with $r+s-1\le n$, $1\le k\le r$, $\rho\in\mathcal{K}_{r+1}$, $\sigma\in\mathcal{K}_s$ and $x_1,\ldots,x_{r+s-1}\in H$,
\begin{align*}
f_{r+s-1}(\partial_k(\rho,\sigma);x_1,\ldots,x_{r+s-1})
=
f_r(\rho;x_1,\ldots,x_{k-1},m_s(\sigma;x_k,\ldots,x_{k+s-1}),x_{k+s},\ldots,x_{r+s-1}),
\end{align*}
\item
for any $r,s\ge1$ with $r+s\le n$, $\rho\in\mathcal{K}_{r+1}$, $\sigma\in\mathcal{K}_{s+1}$, $x_1,\ldots,x_{r+s}\in H$ and $\ell\ge L$,
\[
f_{r+s}(\partial_{r+1}^{\ell}(\rho,\sigma);x_1,\ldots,x_{r+s})
=
f_r(\rho;x_1,\ldots,x_r)f_s(\sigma;x_{r+1},\ldots,x_{r+s}),
\]
\item
for any $x\in H$,
\[
f_1(\ast;x)=f(x),
\]
\item
for any $2\le i\le n$, $1\le k\le i$, $\tau\in\mathcal{K}_{i+1}$ and $x_1,\ldots,x_{i-1}\in H$,
\begin{align*}
f_i(\tau;x_1,\ldots,x_{k-1},\ast,x_k,\ldots,x_{i-1})
=
f_{i-1}(s_k(\tau);x_1,\ldots,x_{k-1},x_k,\ldots,x_{i-1}).
\end{align*}
\end{enumerate}
A pair $(f,(\{f_i\}_i,L))$ of a pointed map and an $A_n$-form on it is called an \textit{$A_n$-map}.
The space of $A_n$-maps from an $A_n$-space $H$ to a topological monoid $G$ is denoted by
\[
\mathcal{A}'_n(H,G)
\subset
\left(\prod_{i=1}^n\Map(\mathcal{K}_{i+1}\times H^i,G)\right)\times[0,\infty].
\]
\end{dfn}

Note that we have two types of $A_n$-maps between topological monoids since any topological monoid $G$ is equipped with the \textit{standard $A_n$-form} $\{m_i\}_i$ given by
\[
m_i(\tau;x_1,\ldots,x_i)=x_1\cdots x_i.
\]
We will see in Proposition \ref{prp_composition_equivalence} that $\mathcal{A}'_n(G,G')$ and $\mathcal{A}_n(G,G')$ are naturally homotopy equivalent.
So the two types of $A_n$-maps are not essentially different in the homotopical sense.

Now we define the composition between $\mathcal{A}'_n(H,G)$ and $\mathcal{A}_n(G,G')$.
Let $\tau\in\mathcal{LT}_n$ and $L\in[0,\infty]$.
Consider the subset $I'(\tau)\subset I(\tau)$ of the internal edges that lie in the shortest path between the $n$-th leaf and the root and
\[
I''_L(\tau)=\{e\in I'(\tau)\mid\ell(e)\ge L\}.
\]
Cutting $\tau$ at the internal edges in $I''_L(\tau)$, we obtain the planar rooted trees $\tau_1,\ldots,\tau_r$ when $\sharp I''_L(\tau)=r-1$ such that $\tau_j$ is closer to the root than $\tau_{j'}$ for $j<j'$.
\begin{figure}[h!]
	\includegraphics[height=36mm]{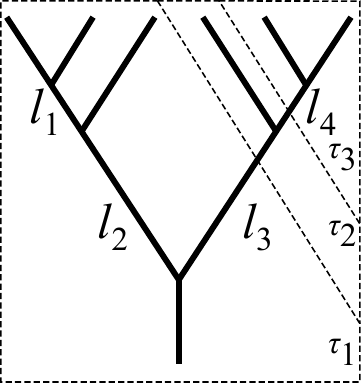}
	\caption{Cutting $\tau$ at the edges in $I''_L(\tau)$ when $l_3,l_4\ge L$.}
\end{figure}
\begin{dfn}
Let $(H,\{m_i\}_i)$ be an $A_n$-space and $G,G'$ be topological monoids.
The composition $(h,(\{h_i\}_i,L+L'))=g\circ f$ of $f=(f,(\{f_i\}_i,L))\in\mathcal{A}'_n(H,G)$ and $g=(g,(\{g_i\}_i,L'))\in\mathcal{A}_n(G,G')$ is defined as follows:
for any $1\le i\le n$, $\tau\in\mathcal{LT}_i$ with $\tau_j\in\mathcal{LT}_{i_k}$ ($k=1,\ldots,r$) obtained by cutting $\tau$ as above and $\mathbf{x}_k\in H^{i_k}$, we define
\begin{align*}
h_i(\tau;\mathbf{x}_1,\ldots,\mathbf{x}_r)
=
g_r(\ell(e_1)-L,\ldots,\ell(e_{r-1})-L;f_{i_1}(\tau_1;\mathbf{x}_1),\ldots,f_{i_r}(\tau_r;\mathbf{x}_r)),
\end{align*}
where $I''_L(\tau)=\{e_1,\ldots,e_r\}$ and $e_j$ is closer to the root than $e_{j'}$ for $j<j'$.
\end{dfn}

By a similar argument to that in \cite[Sections 3 and 4]{MR3491849}, the following theorem holds (see Figure \ref{figure_associativity}).
\begin{thm}
\label{thm_composite_A'}
Let $(H,\{m_i\}_i)$ be an $A_n$-space and $G,G',G''$ be topological monoids.
The composition
\[
\circ\colon\mathcal{A}_n(G,G')\times\mathcal{A}'_n(H,G)\to\mathcal{A}'_n(H,G')
\]
is continuous and the following associativity and unitality hold:
\begin{enumerate}
\item
for any $f\in\mathcal{A}'_n(H,G)$, $g\in\mathcal{A}_n(G,G')$ and $g'\in\mathcal{A}_n(G',G'')$, $g'\circ(g\circ f)=(g'\circ g)\circ f$,
\item
for any $f\in\mathcal{A}'_n(H,G)$, $\operatorname{id}_G\circ f=f$.
\end{enumerate}
\end{thm}
\begin{figure}[h!]
	\label{figure_associativity}
	\includegraphics[height=54mm]{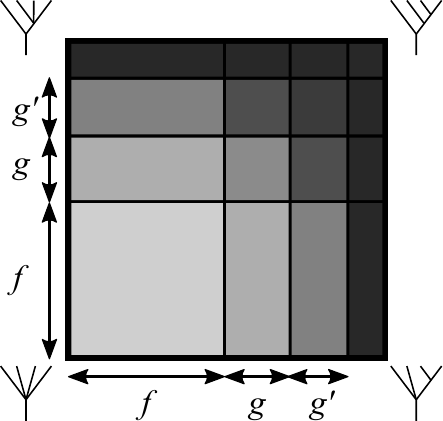}
	\caption{The associativity of the composition $g'\circ g\circ f$ in $\mathcal{K}_4$.}
\end{figure}
An $A_n$-map is said to be \textit{weak $A_n$-equivalence} if the underlying map is a weak homotopy equivalence.
Let $\mathcal{A}'_n(H,G)_{\eq}\subset\mathcal{A}'_n(H,G)$ and $\mathcal{A}_n(G,G')_{\eq}\subset\mathcal{A}_n(G,G')$ denote the subspaces of weak $A_n$-equivalences.
The following proposition can be shown by a proof similar to \cite[Proposition 4.9]{MR3491849}.
\begin{prp}
\label{prp_composition_equivalence}
Let $H$ be an $A_n$-space and $G,G'$ be topological monoids.
Assume that all of them have the homotopy extension property of the basepoint.
The composition with weak $A_n$-equivalences $f\in\mathcal{A}_n'(H,G)_{\eq}$ and $g\in\mathcal{A}_n(G,G')_{\eq}$ induce the weak homotopy equivalences
\begin{align*}
f^{\sharp}\colon\mathcal{A}_n(G,G')\xrightarrow{\simeq}\mathcal{A}_n'(H,G')
\quad
\text{and}
\quad
g_{\sharp}\colon\mathcal{A}'_n(H,G)\xrightarrow{\simeq}\mathcal{A}_n'(H,G').
\end{align*}
In particular, the canonical inclusion $(\id_G)^{\sharp}\colon\mathcal{A}_n(G,G')\to\mathcal{A}_n'(G,G')$ defined to be the composite with the identity is a weak homotopy equivalence for topological monoids $G$ and $G'$.
\end{prp}

\section{Fiberwise $A_n$-spaces and classification theorem}
\label{section_fwAn}

Let $B$ be a space.
We follow the terminology of Crabb--James \cite{MR1646248} as follows.
A \textit{fiberwise space} is just a map $\pi\colon E\to B$ called the projection and a \textit{fiberwise pointed space} is a fiberwise space $\pi\colon E\to B$ equipped with a section $\sigma\colon B\to E$ assigning the basepoint to each fiber.
A \textit{fiberwise maps} and a \textit{fiberwise pointed map} are a map $E\to E'$ compatible with projections and sections in the obvious sense.
Let $\Map_B(E,E')$ and $\Map_B^B(E,E')$ denote the space of fiberwise and fiberwise pointed maps, respectively.
The fiber product $E\times_BE'$ of $E$ and $E'$ is exactly the categorical product with respect to fiberwise (pointed) maps.

A \textit{fiberwise $A_n$-space} $(E,\{m_i\}_i)$ over $B$ is a pair of a fiberwise pointed space $E$ over $B$ and a \textit{fiberwise $A_n$-form} $\{m_i\colon\mathcal{K}_i\times E^i\to E\}_i$, where $E^i$ means the $i$-fold fiber product $E\times_B\cdots\times_BE$.
Here, a fiberwise $A_n$-form is defined in the same way as in Section \ref{section_an-space}.
A fiberwise $A_n$-map between a fiberwise $A_n$-space and a fiberwise topological monoid and between fiberwise topological monoids are also similarly defined.

We take a product-preserving functor $\mathcal{Q}\colon\mathbf{CG}\to\mathbf{CG}$ as assigning a CW replacement.
For example, it is sufficient to take $\mathcal{Q}X$ to be the geometric realization of the singular simplices of a space $X$.
\begin{thm}
\label{thm_classification_fwAn}
Let $E$ be a fiberwise topological monoid over a connected pointed CW complex $B$ such that the projection $E\to B$ is a Hurewicz fibration, the section $B\to E$ has the homotopy extension property and the fiber over the basepoint is $A_n$-equivalent to a topological monoid $G$ where $G$ is a CW complex.
Then there exists a map $B\to B\mathcal{Q}\mathcal{A}_n(G,G)_{\eq}^{\op}$, which we will call a \textit{classifying map}, unique up to homotopy such that the pullback $f^\ast\tilde{E}$ of the universal fiberwise $A_n$-space $\tilde{E}$ is fiberwise $A_n$-equivalent to $E$.
\end{thm}

\begin{rem}
\label{rem_thm_classification_fwAn}
We take the classifying space as $B\mathcal{Q}\mathcal{A}_n(G,G)_{\eq}^{\op}$ rather than the usual one $B\mathcal{Q}\mathcal{A}_n(G,G)_{\eq}$.
The reason for this is that we defined only the composition of $A_n$-maps between topological monoids from the left.
\end{rem}

\begin{proof}
In the classification theorem in \cite{MR2876314}, fiberwise $A_n$-spaces are not assumed to be unital.
But, as stated in \cite[Section 7]{MR3327166}, our theorem is proved by a similar argument except the point that the classifying space there $M_n(G)$ coincides with $B\mathcal{Q}\mathcal{A}_n(G,G)_{\eq}^{\op}$ up to canonical weak homotopy equivalence.
By the same argument as in \cite[Section 5]{MR2876314}, the space
\[
\Prin'(\tilde{E})^{\op}=
\coprod_{b\in M_n(G)}\mathcal{A}'_n(\tilde{E}_b,G)_{\eq}
\]
equipped with an appropriate topology is weakly contractible, where $\tilde{E}_b$ denotes the fiber over the point $b\in M_n(G)$.
Moreover, the composition in Theorem \ref{thm_composite_A'} defines a right action by the grouplike topological monoid $\mathcal{A}_n(G,G)_{\eq}^{\op}$, which is the topological monoid $\mathcal{A}_n(G,G)_{\eq}$ equipped with the opposite multiplication.
Together with Proposition \ref{prp_composition_equivalence}, this implies that $\Prin'(\tilde{E})^{\op}$ is a universal principal fibration for the grouplike topological monoid $\mathcal{A}_n(G,G)_{\eq}^{\op}$.
Thus the space $M_n(G)$ is weakly homotopy equivalent to $B\mathcal{Q}\mathcal{A}_n(G,G)_{\eq}^{\op}$.
\end{proof}

As in \cite[Section 2]{MR2876314}, locally sliceable fiberwise spaces $E,E'$ admit the \textit{fiberwise mapping space} $\map_B(E,E')$, where $\map_B(E,-)$ is the right adjoint functor to $E\times_B-$.
Since it is compatible with pullback by maps of base spaces, one can see that the projection $\map_B(E,E')$ is a Hurewicz fibration when $E\to B$ and $E'\to B$ are Hurewicz fibrations.
Together with it, the fiberwise mapping spaces
\[
\mathscr{A}'_{n}(E,E')=\coprod_{b\in B}\mathcal{A}'_n(E_b,E'_b)
\]
between a fiberwise $A_n$-space $E$ and a fiberwise topological monoid $E'$ over $B$ is naturally topologized and becomes a fiberwise space over $B$.
Also, if the projections $E\to B$ and $E'\to B$ are Hurewicz fibrations and the sections $B\to E$ and $B\to E'$ have the homotopy extension properties, then the projection $\mathscr{A}'_{n}(E,E')\to B$ is a Hurewicz fibration.

The following proposition describes the classifying map of the fiberwise mapping space $\mathscr{A}'_{n}(E,E')$.
Note that when $E'$ is a fiberwise topological monoid with fibers $A_n$-equivalent to a topological monoid $G$, the classifying map is the one $B\to B\mathcal{Q}\mathcal{A}_n(G,G)_{\eq}$ of the principal fibration
\[
\Prin(E')=
\coprod_{b\in B}\mathcal{A}_n(G,E'_b)_{\eq}
\]
equipped with the usual right action $\mathcal{A}_n(G,G)_{\eq}$.

\begin{prp}
\label{prp_classifying_mapping}
Let $E$ be a fiberwise $A_n$-space over $B$ with fibers $A_n$-equivalent to a topological monoid $H$ and $E'$ be a fiberwise topological monoid over $B$ with fibers $A_n$-equivalent to a topological monoid $G$.
Suppose $B,G,H$ are CW complexes and the projections of $E$ and $E'$ are Hurewicz fibrations and the sections of $E$ and $E'$ have the homotopy extension property.
If $E$ is classified by a map $\alpha\colon B\to B\mathcal{Q}\mathcal{A}_n(H,H)_{\eq}^{\op}$ and $E'$ is classified by a map $\beta\colon B\to B\mathcal{Q}\mathcal{A}_n(G,G)_{\eq}$, then the fiberwise mapping space $\mathscr{A}'_{n}(E,E')$ is classified by the composite
\[
B
\xrightarrow{(\alpha,\beta)}
B\mathcal{Q}\mathcal{A}_n(H,H)_{\eq}^{\op}\times B\mathcal{Q}\mathcal{A}_n(G,G)_{\eq}
\xrightarrow{B\mathcal{Q}\Theta}
B\mathcal{Q}\Map(\mathcal{Q}\mathcal{A}_n(H,G),\mathcal{Q}\mathcal{A}_n(H,G))_{\eq}
\]
where $\Map(X,Y)_{\eq}\subset\Map(X,Y)$ denotes the subset of weak equivalences and the homomorphism
\[
\Theta\colon
\mathcal{A}_n(H,H)_{\eq}^{\op}\times\mathcal{A}_n(G,G)_{\eq}
\to
\Map(\mathcal{A}_n(H,G),\mathcal{A}_n(H,G))_{\eq}
\]
is given by $\Theta(f,g)(h)=g\circ h\circ f$.
\end{prp}

\begin{proof}
Let $\tilde{E}\to B\mathcal{Q}\mathcal{A}_n(H,H)_{\eq}^{\op}$ and $\tilde{E}'\to B\mathcal{Q}\mathcal{A}_n(G,G)_{\eq}$ be the universal fiberwise $A_n$-spaces.
We define $\mathcal{G}=\mathcal{A}_n(H,H)_{\eq}^{\op}\times\mathcal{A}_n(G,G)_{\eq}$ and $\mathcal{B}=B\mathcal{QG}=B\mathcal{Q}\mathcal{A}_n(H,H)_{\eq}^{\op}\times B\mathcal{Q}\mathcal{A}_n(G,G)_{\eq}$ for simplicity.
Consider the commutative diagram
\[
\xymatrix{
\mathscr{A}'_{n}(\tilde{E},\tilde{E}') \ar[d]
& B(\Prin'(\tilde{E})^{\op}\times\Prin(\tilde{E}'),\mathcal{QG},\mathcal{Q}\mathcal{A}_n(H,G))
	\ar[l] \ar[r] \ar[d]
& B(\ast,\mathcal{QG},\mathcal{Q}\mathcal{A}_n(H,G)) \ar[d] \\
\mathcal{B}
& B(\Prin'(\tilde{E})^{\op}\times\Prin(\tilde{E}'),\mathcal{QG},\ast) \ar[l] \ar[r]
& B(\ast,\mathcal{QG},\ast)
}
\]
where the left arrows are induced from the compositions of $A_n$-maps and the squares are homotopy pullback by Lemma \ref{lem_bar_quasifib}.
Since the vertical maps are quasifibrations and the bottom arrows are  weak homotopy equivalences, the top horizontal arrows are also weak homotopy equivalences.
Inverting the bottom left arrow, we obtain the composite
\[
\mathcal{B}=B\mathcal{QG}
\to
B(\Prin'(\tilde{E})^{\op}\times\Prin(\tilde{E}'),\mathcal{QG},\ast)
\to
B\mathcal{QG},
\]
which is indeed homotopic to the identity map since it is the classifying map of the universal principal fibration.
Consider the homotopy pullback square
\[
\xymatrix{
B(\ast,\mathcal{QG},\mathcal{Q}\mathcal{A}_n(H,G)) \ar[r] \ar[d]
& B(\ast,\mathcal{Q}\Map(\mathcal{Q}\mathcal{A}_n(H,G),\mathcal{Q}\mathcal{A}_n(H,G))_{\eq},
	\mathcal{Q}\mathcal{A}_n(H,G))
	\ar[d] \\
B\mathcal{QG} \ar[r]^-{B\mathcal{Q}\Theta}
& B\mathcal{Q}\Map(\mathcal{Q}\mathcal{A}_n(H,G),\mathcal{Q}\mathcal{A}_n(H,G))_{\eq}.
}
\]
This implies that the Hurewicz fibration $\mathscr{A}'_{n}(\tilde{E},\tilde{E}')\to\mathcal{B}$ is classified by $B\mathcal{Q}\Theta$.
Since $\mathscr{A}'_{n}(E,E')$ is the pullback of $\mathscr{A}'_n(\tilde{E},\tilde{E}')\to\mathcal{B}$ by the map $(\alpha,\beta)\colon B\to\mathcal{B}$, the proposition follows.
\end{proof}

\section{$N_k(\ell)$-map}
\label{section_normal}

We introduce $N_k(\ell)$-map mimicking Definition \ref{dfn_crossed} of crossed module as follows.

\begin{dfn}
\label{dfn_(k,l)-normal}
Let $f\colon H\to G$ be a homomorphism between topological groups and $k,\ell$ be positive integers or infinity.
We say the homomorphism $f$ is an \textit{$N_k(\ell)$-map} if an $A_k$-map $\rho\colon G\to\mathcal{A}_\ell(H,H)$ is given and the following conditions hold:
\begin{enumerate}
\item
the composite of $A_k$-maps
\[
H
\xrightarrow{f}
G
\xrightarrow{\rho}
\mathcal{A}_\ell(H,H)
\]
is homotopic to $\conj\colon H\to\mathcal{A}_\ell(H,H)$ ($\conj(h)(x)=hxh^{-1}$) as an $A_k$-map,
\item
the pair of a map
\[
\text{point}
\to
\mathcal{A}_\ell(H,G),
\quad
\ast\mapsto f
\]
and the $A_k$-map
\[
\theta\colon G\to\Map(\mathcal{A}_\ell(H,G),\mathcal{A}_\ell(H,G))_{\eq},
\quad
\theta(g)(\alpha)=\conj(g)\circ\alpha\circ\rho(g^{-1}).
\]
extends to an $A_k$-equivariant map $(G,\ast)\to(\Map(\mathcal{A}_\ell(H,G),\mathcal{A}_\ell(H,G))_{\eq},\mathcal{A}_\ell(H,G))$,
\item
the composite of the $A_k$-form of the previous $A_k$-equivariant map with $f\colon(H,\ast)\to(G,\ast)$ is homotopic to the trivial one coming from the equality $\conj(f(h))\circ f\circ\conj(h^{-1})=f$ for any $h\in H$.
\end{enumerate}
\end{dfn}

This definition could be extended to an $A_n$-map $f$.
To do so, it is necessary to determine how good coherence can be guaranteed for the action of $H$ on $\Map(\mathcal{A}_\ell(H,G),\mathcal{A}_\ell(H,G))_{\eq}$ in the condition (3).
We avoid this problem and concentrate on homomorphisms here.

If a homomorphism $f\colon H\to G$ is an $N_k(\ell)$-map and $k'\le k,\ell'\le\ell$, then $f$ is obviously an $N_{k'}(\ell')$-map.

If $f\colon H\to G$ is a crossed module, then $f$ is obviously an $N_\infty(\infty)$-map.
This is the only case when we can verify the condition of Definition \ref{dfn_(k,l)-normal} directly since it contains very complicated higher homotopical conditions in the present work.
The following theorem provides us the way to check the conditions from the obstruction theoretic point of view.

\begin{thm}
\label{thm_main}
Let $f\colon H\to G$ be a homomorphism between topological groups $G,H$, of which the underlying spaces are CW complexes, and $k,\ell$ are non-negative integers.
Then the homomorphism $f$ is an $N_k(\ell)$-map if and only if there exists a fiberwise $A_\ell$-space $E$ over $B_kG$ and fiberwise $A_\ell$-maps $\phi\colon(B_kf)^\ast E\to E_kH\times_{\conj}H$ over $B_kH$ and $\psi\colon E\to E_kG\times_{\conj}G$ over $B_kG$ satisfying the following conditions:
\begin{enumerate}
\item
$\phi$ restricts to a weak $A_\ell$-equivalence on each fiber,
\item
for the restrictions to the fiber over the basepoint $\phi_\ast,\psi_\ast$, the composite $f\circ\phi_\ast\colon E_\ast\to G$, where $E_\ast$ is the fiber of $E$ over the basepoint, is homotopic to $\psi_\ast$ as an $A_\ell$-map,
\item
the composite of $\phi$ and the induced map $E_kH\times_{\conj}H\to E_kH\times_{\conj\circ f}G$ of $f$ is homotopic to $(B_kf)^\ast\psi$ as a fiberwise $A_\ell$-map.
\end{enumerate}
\end{thm}

Actually, the condition (2) is immediately implied by (3).
But we dare to write in this way in order to clarify the correspondence of the conditions between this theorem and Definition \ref{dfn_(k,l)-normal}.

Before proving the proof, let us see the following example to illustrate the meaning of the fiberwise $A_\ell$-space $E$ in the theorem.

\begin{ex}
Suppose that $H\subset G$ is a closed normal subgroup.
The fiberwise topological group
\[
E=EG\times_{\conj}H
\]
is constructed by the conjugation action of $G$ on $H$.
Then we have the obvious factorization
\[
EH\times_{\conj}H
\to
E
\to
EG\times_{\conj}G
\]
of the natural inclusion $EH\times_{\conj}H\to EG\times_{\conj}G$.
Let $\phi\colon E|_{BH}\to EH\times_{\conj}H$ be the isomorphism of fiberwise topological groups over $BH$ from $E$ restricted to the subspace $BH\subset BG$ and $\psi\colon E\to EG\times_{\conj}G$ be the inclusion.
This factorization satisfies the conditions in the theorem for $k=\ell=\infty$.
\end{ex}

\begin{proof}[Proof of Theorem \ref{thm_main}]
Suppose the condition (1) in the theorem.
Let $\rho_0'\colon B_kG\to B\mathcal{Q}\mathcal{A}_\ell(H,H)_{\eq}^{\op}$ be the classifying map of $\Prin'(E)^{\op}$ as in Theorem \ref{thm_classification_fwAn}.
By the assumption (1), the composite $\rho_0'\circ B_kf$ is homotopic to $\iota_k\circ B_k(\conj\circ\text{(inversion)})\colon B_kH\to B\mathcal{Q}\mathcal{A}_\ell(H,H)_{\eq}$.
Let $\rho\colon G\to\mathcal{A}_\ell(H,H)_{\eq}^{\op}$ be the composite of the inversion and the $A_k$-map adjoint to $\rho_0'$.
Then the composite $\rho\circ f$ is homotopic to $\conj\colon H\to\mathcal{A}_\ell(H,H)$ as an $A_k$-map.
This is the condition (1) in Definition \ref{dfn_(k,l)-normal}.

Conversely, suppose the condition (1) in Definition \ref{dfn_(k,l)-normal}.
Let $E$ be the pullback of the universal fiberwise $A_n$-space by the composite
\[
B_kG
\xrightarrow{B_k\text{(inversion)}}
B_kG^{\op}
\xrightarrow{B_k\rho}
B_k\mathcal{Q}\mathcal{A}_\ell(H,H)_{\eq}^{\op}
\xrightarrow{\iota_k}
B\mathcal{Q}\mathcal{A}_\ell(H,H)_{\eq}^{\op}.
\]
Then the pullback $(B_kf)^\ast E$ is weakly fiberwise $A_\ell$-equivalent to the fiberwise topological group $E_kH\times_{\conj}H$.
This is the condition (1) in the theorem.

Suppose the condition (1), (2) and (3) in the theorem.
The fiberwise $A_\ell$-map $\psi$ defines a section of the fiberwise mapping space $\mathscr{A}'_{\ell}(E,E_kG\times_{\conj}G)$ over $B_kG$.
By the map $\phi$, we can identify the fiber of $\mathscr{A}_\ell(E,E_kG\times_{\conj}G)$ over the base point with $\mathcal{A}_\ell(H,G)$ of which the basepoint is $f$.
By Proposition \ref{prp_classifying_mapping}, the fiberwise space $\mathscr{A}'_{\ell}(E,E_kG\times_{\conj}G)$ is classified by the composite
\[
B_kG
\xrightarrow{(\rho_0',\iota_k\circ B_k\conj)}
B\mathcal{Q}\mathcal{A}_\ell(H,H)_{\eq}^{\op}\times B\mathcal{Q}\mathcal{A}_\ell(G,G)_{\eq}
\xrightarrow{B\mathcal{Q}\Theta}
B\mathcal{Q}\Map(\mathcal{Q}\mathcal{A}_\ell(H,G),\mathcal{Q}\mathcal{A}_\ell(H,G))_{\eq}.
\]
The section induced from $\psi$ determines the lift
\[
\Psi'\colon B_kG
\to
B\mathcal{Q}\Map_\ast(\mathcal{Q}\mathcal{A}_\ell(H,G),\mathcal{Q}\mathcal{A}_\ell(H,G))_{\eq}
\]
of the previous composite along the canonical map
\[
B\mathcal{Q}\Map_\ast(\mathcal{Q}\mathcal{A}_\ell(H,G),\mathcal{Q}\mathcal{A}_\ell(H,G))_{\eq}
\to
B\mathcal{Q}\Map(\mathcal{Q}\mathcal{A}_\ell(H,G),\mathcal{Q}\mathcal{A}_\ell(H,G))_{\eq}.
\]
Let $\Psi\colon G\to\Map_\ast(\mathcal{Q}\mathcal{A}_\ell(H,G),\mathcal{Q}\mathcal{A}_\ell(H,G))_{\eq}$ be the $A_k$-map adjoint to $\Psi'$.
It defines the obvious $A_k$-equivariant map
\[
\Psi\colon(G,\ast)\to(\Map(\mathcal{Q}\mathcal{A}_\ell(H,G),\mathcal{Q}\mathcal{A}_\ell(H,G))_{\eq},\mathcal{Q}\mathcal{A}_\ell(H,G)),
\]
where the underlying map $\ast\to\mathcal{Q}\mathcal{A}_\ell(H,G)$ is $\ast\mapsto f$.
By the argument so far, the underlying $A_k$-map of $\Psi$ is homotopic to the composite of $A_k$-maps
\[
G
\xrightarrow{\theta}
\Map(\mathcal{A}_\ell(H,G),\mathcal{A}_\ell(H,G))_{\eq}
\xrightarrow{\mathcal{Q}}
\Map(\mathcal{Q}\mathcal{A}_\ell(H,G),\mathcal{Q}\mathcal{A}_\ell(H,G))_{\eq}.
\]
This homotopy determines an $A_k$-equivariant map
\[
(G,\ast)\to(\Map(\mathcal{A}_\ell(H,G),\mathcal{A}_\ell(H,G))_{\eq},\mathcal{A}_\ell(H,G))
\]
with the underlying $A_k$-map $\theta$.
To be precise, it is determined up to homotopy.
Also, the composite
\[
B_kH
\xrightarrow{B_kf}
B_kG
\xrightarrow{\Psi'}
B\mathcal{Q}\Map_\ast(\mathcal{Q}\mathcal{A}_\ell(H,G),\mathcal{Q}\mathcal{A}_\ell(H,G))_{\eq}
\]
is homotopic to the map adjoint to the $A_k$-map (actually a homomorphism)
\[
H
\to
\mathcal{Q}\Map_\ast(\mathcal{Q}\mathcal{A}_\ell(H,G),\mathcal{Q}\mathcal{A}_\ell(H,G))_{\eq}
\]
induced from the conjugation on $H$ and the conjugation through $f$ on $G$ by the assumption (3).
Thus the condition (2) and (3) in Definition \ref{dfn_(k,l)-normal} is verified.

Conversely, suppose the conditions (1), (2) and (3) in Definition \ref{dfn_(k,l)-normal}.
By Proposition \ref{prp_classifying_mapping}, $B_k\theta$ classifies the fiberwise mapping space $\mathscr{A}'_{\ell}(E,E_kG\times_{\conj}G)$.
The bar construction of the $A_k$-equivariant map
\[
(G,\ast)\to(\Map(\mathcal{A}_\ell(H,G),\mathcal{A}_\ell(H,G))_{\eq},\mathcal{A}_\ell(H,G))
\]
defines a map
\[
B_kG=B_k(\ast,G,\ast)
\to
B_k(\ast,\mathcal{Q}\Map(\mathcal{Q}\mathcal{A}_\ell(H,G),\mathcal{Q}\mathcal{A}_\ell(H,G))_{\eq},
\mathcal{Q}\mathcal{A}_\ell(H,G))
\]
of which the composition with the projection onto $B\mathcal{Q}\Map(\mathcal{Q}\mathcal{A}_\ell(H,G),\mathcal{Q}\mathcal{A}_\ell(H,G))_{\eq}$ is homotopic to $B_k\theta$.
Then $\mathscr{A}'_{\ell}(E,E_kG\times_{\conj}G)$ admits a section which restricts to $f$ up to the canonical $A_k$-equivalence $E_\ast\simeq H$.
This verifies the condition (2) in the theorem.
Moreover, since the $A_k$-equivariant map restricts to the trivial one on $H$ as supposed, we can verify the condition (3).
\end{proof}

Let us consider the special case when $k=\ell=1$.
For pointed spaces $A,B$ with nondegenerate basepoints and a topological group $G$, we have the natural split exact sequence
\[
1
\to
[A\wedge B,G]
\xrightarrow{q^\ast}
[A\times B,G]
\xrightarrow{i^\ast}
[A\vee B,G]
\to
1,
\]
where $i\colon A\vee B\to A\times B$ is the inclusion and $q\colon A\times B\to A\wedge B$ is the quotient map. 
The \textit{Samelson product} $\langle f,g\rangle\in[A\wedge B,G]$ of based maps $f\colon A\to G$ and $g\colon B\to G$ is the homotopy class determined by the commutator map
\[
A\times B\to G,
\quad
(a,b)\mapsto f(a)g(b)f(a)^{-1}g(b)^{-1}
\]
and the previous split exact sequence.

\begin{thm}
\label{thm_(1,1)}
Let $f\colon H\to G$ be a homomorphism between topological groups $G,H$, of which the underlying spaces are CW complexes.
The map $f$ is an $N_1(1)$-map if and only if there exists a pointed map $\rho'\colon G\wedge H\to H$ satisfying the following conditions:
\begin{enumerate}
\item
the composite $\rho'\circ(f\wedge\operatorname{id}_H)\colon H\wedge H\to H$ is homotopic to the Samelson product $\langle\operatorname{id}_H,\operatorname{id}_H\rangle$,
\item
the composite $f\circ\rho'\colon G\wedge H\to G$ is homotopic to the Samelson product $\langle\operatorname{id}_G,f\rangle$
\item
the composite of above two homotopies is homotopic to the stationary homotopy on $\langle f,f\rangle$.
\end{enumerate}
\[
\xymatrix{
H\wedge H
	\ar[rr]^-{\langle\operatorname{id}_H,\operatorname{id}_H\rangle}
	\ar[d]_-{f\wedge\operatorname{id}_H}
	& & H \ar[d]^-{f} \\
G\wedge H
	\ar@{.>}[urr]^-{\rho'}
	\ar[rr]_-{\langle\operatorname{id}_G,f\rangle}
	& & G
}
\]
\end{thm}

\begin{proof}
For given $\rho'\colon G\wedge H\to H$, define
\[
\rho\colon G\to\Map_\ast(H,H),
\quad
\rho(g)(h)=\rho'(g,h)h.
\]
Then the condition (1) is equivalent to the one that $\rho\circ f$ is homotopic to $\conj\colon H\to\Map_\ast (H,H)$.
The latter condition is nothing but the condition (1) in Definition \ref{dfn_(k,l)-normal}.
The condition (2) is equivalent to the one that the map $(g,h)\mapsto f(\rho(g)(h))$ is homotopic to the map $(g,h)\mapsto gf(h)g^{-1}$.
The latter condition is equivalent to the condition (2) in Definition \ref{dfn_(k,l)-normal}.
The equivalence between the condition (3) and the condition (3) in Definition \ref{dfn_(k,l)-normal} follows similarly.
\end{proof}

\begin{rem}
\label{rem_James}
From this result a homomorphism $f\colon H\to G$ is an $N_1(1)$-map if and only if $f$ is homotopy normal in the sense of McCarty \cite{MR176471}.
The condition (2) is exactly the homotopy normality of James \cite{MR233376}.
\end{rem}

\section{$C(k,\ell)$-space}
\label{section_C(k,l)}

A \textit{$C(k,\ell)$-space} is a topological monoid with certain higher homotopy commutativity introduced in 
\cite{MR2678992}.
We do not recall the precise definition but instead quote the following characterization \cite[Theorem 8.3]{MR3491849}, where the equivalence with fourth condition follows from Theorem \ref{thm_classification_fwAn}.

\begin{thm}
\label{thm_C(k,l)_char}
Let $G$ be a topological group, of which the underlying space is a CW complex.
Then the following conditions are equivalent:
\begin{enumerate}
\item
$G$ is a $C(k,\ell)$-space,
\item
the wedge sum of the inclusions $(\iota_k,\iota_\ell)\colon B_kG\vee B_\ell G\to BG$ extends over the product $B_kG\times B_\ell G$,
\item
the homomorphism $\conj\colon\mathcal{A}_\ell(G,G)$ is homotopic to the constant map to the identity as $A_k$-map,
\item
the fiberwise topological group $E_kG\times_{\conj}G$ over $B_kG$ is fiberwise $A_\ell$-equivalent to the trivial fiberwise topological group $B_kG\times G$.
\end{enumerate}
\end{thm}

Note that, in particular, $G$ is a $C(\infty,\infty)$-space if and only if $BG$ is an $H$-space.

The following is an easy application of the previous theorem.

\begin{thm}
\label{thm_C(k,l)_(k,l)normal}
Let $H$ and $G$ be topological groups of which the underlying spaces are CW complexes and $f\colon H\to G$ be a homomorphism.
If $H$ and $G$ are $C(k,\ell)$-spaces for some $k,\ell\le n$, then $f$ is an $N_k(\ell)$-map.
\end{thm}
\begin{proof}
By assumption and Theorem \ref{thm_C(k,l)_char}, the fiberwise topological groups $E_kH\times_{\conj}H$ and $E_kG\times_{\conj}G$ are fiberwise $A_\ell$-equivalent to the trivial bundles.
Consider the trivial fiberwise topological group $E=B_kG\times H$.
Then we can verify the conditions in Theorem \ref{thm_main}.
Thus $f$ is an $N_k(\ell)$-map.
\end{proof}

For the homomorphism $H\to\ast$ to the trivial group, we obtain the following.

\begin{thm}
\label{thm_Nkl_to_trivial}
Let $H$ be a topological group of which the underlying space is a CW complex.
The homomorphism $H\to\ast$ to the trivial group is an $N_k(\ell)$-map if and only if $H$ is a $C(k,\ell)$-space.
\end{thm}

\begin{proof}
Note that $B_k\ast=\ast$ for any $k$.
Then by Theorem \ref{thm_main}, the homomorphism $H\to\ast$ is an $N_k(\ell)$-map if and only if the fiberwise topological monoid $E_kH\times_{\conj}H$ is fiberwise $A_\ell$-equivalent to the trivial fiberwise topological group $B_kH\times H$.
Thus the theorem follows from Theorem \ref{thm_C(k,l)_char}.
\end{proof}

\section{$H$-structure on quotient space}
\label{section_quotient}

Since the homotopy quotient of a crossed module is known to be a topological monoid, we expect a similar result for $N_k(\ell)$-maps.
Let us investigate the existence of an $H$-structure on the homotopy quotient.
\begin{prp}
\label{prp_quotient_Ckl}
Let $f\colon H\to G$ be a homomorphism between topological groups of which the underlying spaces are CW complexes.
Suppose $f$ is an $N_k(\ell)$-map.
Then, the topological group $F$ with classifying space $BF\simeq B(\ast,H,G)=EH\times_fG$ is a $C(k,\ell)$-space.
\end{prp}

Note that the topological group $F$ is $A_\infty$-equivalent to the Moore based loop space of $BF$.
This is a key proposition in our study of $H$-structures on $EH\times_fG$.

\begin{proof}
We may suppose that $f$ is a Hurewicz fibration and $F=\ker f$.
Let $i\colon F\to H$ denote the inclusion.
Note that the results in \cite[Sections 3 and 4]{MR3327166} for $A_n$-spaces are similarly verified for fiberwise $A_n$-spaces.
By \cite[Proposition 4.1]{MR3327166}, we may suppose that there exists a fiberwise $A_\ell$-space $E$ over $B_kG$, a fiberwise $A_\ell$-equivalence $\phi\colon(B_kf)^\ast E\to E_kH\times_{\conj}H$ and a fiberwise $A_\ell$-homomorphism $\psi\colon E\to E_kG\times_{\conj}G$ as in Theorem \ref{thm_main}.
Let $E'$ be the fiberwise homotopy fiber of $\psi$ over $B_kG$.
Then by \cite[Theorem 3.1]{MR3327166} (the pullback of $A_n$-homomorphisms), $E'$ admits the canonical structure of a fiberwise $A_\ell$-space.
Applying \cite[Theorem 3.3]{MR3327166} (the universal property of the pullback of $A_n$-homomorphisms), a fiberwise $A_\ell$-map $\chi\colon(B_kf)^\ast E'\to E_kH\times_{H}F$ is induced as in the following diagram:
\[
\xymatrix{
(B_kf)^\ast E' \ar[r] \ar@{.>}[d]_-{\chi}
	& (B_kf)^\ast E \ar[r]^-{\psi_{B_kH}} \ar[d]^-{\phi}
	& E_kH\times_{\conj\circ f}G \ar@{=}[d] \\
E_kH\times_{H}F \ar[r]
	& E_kH\times_{\conj}H \ar[r]
	& E_kH\times_{\conj\circ f}G
}
\]
We can see that $\chi$ is a fiberwise $A_\ell$-equivalence.
Pulling buck along the map $B_ki\colon B_kF\to B_kH$, we obtain the fiberwise $A_\ell$-equivalence
\[
(B_ki)^\ast\chi
\colon
(B_k(i\circ f))^\ast E'
\xrightarrow{\simeq}
E_kF\times_{\conj}F.
\]
Since $i\circ f$ is the constant homomorphism, $E_kF\times_HF$ is fiberwise $A_\ell$-equivalent to the trivial fiberwise topological group $B_kF\times F$.
Thus, it follows from Theorem \ref{thm_C(k,l)_char} that $F$ is a $C(k,\ell)$-space.
\end{proof}

\begin{thm}
\label{thm_quotient_Hstr}
Let $f\colon H\to G$ be a homomorphism between topological groups of which the underlying spaces are CW complexes and $k\ge1$.
Suppose that $f$ is an $N_k(k)$-map and the LS category of the homotopy quotient $EH\times_fG$ is estimated as $\cat(EH\times_fG)\le k$.
Then $EH\times_fG$ is an $H$-space.
\end{thm}

\begin{proof}
Take $F$ as in the proof of Proposition \ref{prp_quotient_Ckl} so that the homotopy equivalence $BF\simeq EH\times_fG$ and the estimate $\cat BF\le k$ hold.
Then, as in \cite{MR1642747}, the inclusion $B_kF\to BF$ admits a homotopy section $s\colon BF\to B_kF$.
Since $F$ is a $C(k,k)$-space, it follows from the existence of $s$ and Theorem \ref{thm_C(k,l)_char} that $BF$ is an $H$-space.
\end{proof}

\begin{rem}
As a consequence of Theorem \ref{thm_Nkl_to_trivial}, one cannot expect that $N_k(\ell)$-map implies any higher homotopy associativity of the homotopy quotient $H$-space.
In order to guarantee higher homotopy associativity, we need even higher homotopy normality.
\end{rem}

\begin{rem}
Although Theorem \ref{thm_quotient_Hstr} seems theoretically important, we do not have its good application at this point.
For example, we shall give a sufficient condition in Section \ref{section_Lie} below that the inclusion $\SU(m)\to\SU(n)$ is $p$-locally an $N_k(\ell)$-map.
But it will not give a new result on the existence of an $H$-structure since the quotient $\SU(n)/\SU(m)$ is $p$-locally homotopy equivalent to a product of odd dimensional spheres under the assumption there.
\end{rem}

We have only shown the existence of an $H$-structure on the homotopy quotient in the previous theorem.
We do not try to answer the following natural question in the present work.

\begin{problem}
Suppose a homomorphism $f\colon H\to G$ is an $N_k(k)$-map and $\cat(EH\times_fG)\le k$.
Construct a canonical $H$-structure on $EH\times_fG$.
If it can be done, is the natural map $G\to EH\times_fG$ an $H$-map?
\end{problem}

\section{Fiberwise projective spaces}
\label{section_projective}

Let us recall the  fiberwise projective spaces of topological monoids.

\begin{dfn}
The \textit{fiberwise $n$-th projective space} $\mathscr{B}_nE$ of a fiberwise topological monoid $E\to B$ is defined to be the quotient
\[
\mathscr{B}_nE
=
\left(\coprod_{0\le i\le n}\Delta^i\times E^i\middle)\right/{\sim}
\]
by the usual simplicial relation, where $E^i$ denotes the $i$-fold fiber product.
In particular, $\mathscr{B}E=\mathscr{B}_\infty E$ is called the \textit{fiberwise classifying space}.
Let $\iota_n\colon\mathscr{B}_nE\to\mathscr{B}E$ denote the canonical inclusion.
As in Sections \ref{section_monoid}, the fiberwise projective space induce the fiberwise map
\[
\mathscr{B}_n\colon\mathscr{A}_n(E,E')\to\map_B^B(\mathscr{B}_nE,\mathscr{B}_nE')
\]
to the fiberwise mapping space of fiberwise pointed maps for fiberwise topological monoids $E$ and $E'$.
\end{dfn}

The fiberwise $n$-th projective space for a fiberwise $A_n$-space is similarly constructed as in \cite{MR2665235}.

\begin{dfn}
The \textit{fiberwise $n$-th projective space} $\mathscr{B}^{\mathcal{K}}_nE$ of a fiberwise $A_n$-space $E\to B$ is the quotient
\[
\mathscr{B}^{\mathcal{K}}_nE
=
\left(\coprod_{0\le i\le n}\mathcal{K}_{i+2}\times E^i\middle)\right/{\sim}
\]
constructed as in \cite[Construction 8]{MR0158400}.
\end{dfn}

The following is a variant of Theorem \ref{thm_Tsu16} in the category of fiberwise spaces.

\begin{thm}
\label{thm_Tsu16_fw}
Let $E$ and $E'$ be fiberwise topological monoids over a CW complex $B$ and suppose that their fiberwise basepoints of $E$ and $E'$ have the fiberwise homotopy extension property, their projections are Hurewicz fibrations and $E'$ is grouplike on each fiber.
Then, the composite of fiberwise maps
\[
\mathscr{A}_n(E,E')
\xrightarrow{\mathscr{B}_n}
\map_B^B(\mathscr{B}_nE,\mathscr{B}_nE')
\xrightarrow{(\iota_n)_\sharp}
\map_B^B(\mathscr{B}_nE,\mathscr{B}E')
\]
is a weak homotopy equivalence in the category of spaces.
\end{thm}
\begin{proof}
In this setting, the projections $\mathscr{A}_n(E,E')\to B$ and $\map_B^B(\mathscr{B}_nE,\mathscr{B}E')\to B$ are Hurewicz fibrations.
Thus the given composite is a weak homotopy equivalence since it restricts to a weak homotopy equivalence on each fiber by Theorem \ref{thm_Tsu16}. 
\end{proof}

Although the following result is well-known, we give a proof here.

\begin{prp}
Let $G$ be a topological group.
Then, the fiberwise classifying space $EG\times_GBG$ is homeomorphic to $BG\times BG$ through a fiberwise pointed map over $BG$, where the fiberwise basepoint of $BG\times BG$ is given by the diagonal map $BG\to BG\times BG$.
\end{prp}

\begin{proof}
The homeomorphism $EG\times_GBG\to BG\times BG$ is induced from the geometric realization of the simplicial map $\{G^i\times G^i\to G^i\times G^i\}_i$ given by the homeomorphisms
\begin{align*}
&(g_1,g_2,\ldots,g_i;x_1,x_2,\ldots,x_i)\\
&\mapsto
(g_1,g_2,\ldots,g_i;
	(g_1\cdots g_i)x_1(g_2\cdots g_i)^{-1},
	(g_2\cdots g_i)x_2(g_3\cdots g_i)^{-1},
	\ldots,
	g_ix_i),
\end{align*}
where $EG\times_GBG$ and $BG\times BG$ are cosidered to be the geometric realizations of appropriate simplicial spaces $\{G^i\times G^i\}_i$ with different simplicial structures.
\end{proof}

The following theorem provides a criterion for non-$N_k(\ell)$-map.

\begin{thm}
\label{thm_Nkl_fwproj}
Let $f\colon H\to G$ be a homomorphism between topological groups of which the underlying spaces are CW complexes and $k,\ell\ge1$.
Then, if $f$ is an $N_k(\ell)$-map, then there exist fiberwise pointed maps
\[
E_kH\times_HB_{\ell}H
\xrightarrow{\Phi}
\mathcal{E}
\xrightarrow{\Psi}
EG\times_GBG
\]
for some fiberwise pointed space $\mathcal{E}$ over $B_kG$, where $\Phi$ is precisely a fiberwise pointed map as a map $E_kH\times_HB_{\ell}H\to(B_kf)^\ast\mathcal{E}$, satisfying the following conditions:
\begin{enumerate}
\item
$\Phi$ restricts to a weak homotopy equivalence on each fiber,
\item
the restriction of the composite $\Psi\circ\Phi$ to the fiber over the basepoint is homotopic to $\iota_\ell\circ B_\ell f\colon B_\ell H\to BG$,
\item
the composite $((B_kf)^\ast\Psi)\circ\Phi\colon E_kH\times_HB_\ell H\to E_kH\times_HBG$ of $\Phi$ and the pullback of $\Psi$ by the map $B_kf\colon B_kH\to B_kG$ is homotopic to the map induced from $f$ as a fiberwise pointed map.
\end{enumerate}
\end{thm}

As in Theorem \ref{thm_main}, the condition (2) which immediately follows from (3) is contained.

\begin{rem}
The fiberwise pointed space $\mathcal{E}$ is obtained as the fiberwise projective space of the fiberwise $A_n$-space in Theorem \ref{thm_main}.
So the converse of this theorem should hold as well.
But, to verify it, we need to generalize Theorem \ref{thm_Tsu16_fw} for fiberwise $A_n$-maps from fiberwise $A_n$-spaces in a way compatible with the composition of fiberwise $A_n$-maps between fiberwise topological monoids.
We leave this problem for now.
\end{rem}

\begin{proof}
Suppose $f\colon H\to G$ is an $N_k(\ell)$-map.
Then we have a fiberwise $A_\ell$-space $E$ over $B_kG$ and fiberwise $A_\ell$-maps $\phi\colon E_kH\times_{\conj}H\to(B_kf)^\ast E$ and $\psi\colon E\to E_kG\times_{\conj}G$ as in Theorem \ref{thm_main}.
Again by \cite[Proposition 4.1]{MR3327166}, we may suppose that $\phi$ and $\psi$ are fiberwise $A_\ell$-homomorphisms.
Moreover, we may suppose the composite of $\phi$ and the fiberwise homomorphism $E_kH\times_{\conj}H\to E_kH\times_{\conj\circ f}G$ is homotopic to $\psi$ through fiberwise $A_\ell$-homomorphisms by this argument. 
Let $\mathcal{E}=\mathscr{B}^{\mathcal{K}}_\ell E$.
Note that a fiberwise $A_\ell$-homomorphism induces a fiberwise pointed map between the fiberwise projective spaces in the obvious manner.
Then we have fiberwise $A_\ell$-maps $\Phi'\colon\mathscr{B}^{\mathcal{K}}_\ell(E_kH\times_{\conj}H)\to(B_kf)^\ast\mathcal{E}$, which is the homotopy inverse of the induced map of $\phi$, $\Psi'\colon\mathcal{E}\to\mathscr{B}^{\mathcal{K}}_\infty(EG\times_{\conj}G)$, which is the induced map of $\psi$, and $F'\colon\mathscr{B}^{\mathcal{K}}_\ell(E_kH\times_{\conj}H)\to\mathscr{B}^{\mathcal{K}}_\infty(E_kH\times_{\conj\circ f}G)$, which is the induced map of $f$.
We also note that there exists a canonical homotopy commutative diagram of fiberwise pointed spaces
\[
\xymatrix{
\mathscr{B}^{\mathcal{K}}_\ell(E_kH\times_{\conj}H) \ar[r]^-{F'} \ar[d]_-{\simeq}
& \mathscr{B}^{\mathcal{K}}_\infty(E_kH\times_{\conj\circ f}G) \ar[d]^-{\simeq} \\
E_kH\times_HB_\ell H \ar[r]_-{F}
& E_kH\times_HBG
}
\]
where $F$ is the induced map of $f$ and the vertical maps are fiberwise pointed homotopy equivalences.
This follows from the classification theorem of fiberwise principal bundles and the fact that both $\mathscr{B}^{\mathcal{K}}_\ell(E_kH\times_{\conj}H)$ and $E_kH\times_HB_\ell H$ are characterized by the fiberwise variant of Ganea's pullback-pushuout construction.
Composing the above fiberwise pointed homotopy equivalences, we obtain the fiberwise pointed maps $\Phi$ and $\Psi$ satisfying the desired conditions.
\end{proof}

\section{$p$-local normality of Lie groups}
\label{section_Lie}

By the result of James \cite{MR233376} and Remark \ref{rem_James}, the inclusion $\SU(m)\to\SU(n)$ ($2\le m<n$) is not ($2$-locally) an $N_1(1)$-map.
But one might expect some normality when the spaces are localized at a large prime $p$.
In the rest of this section, we assume $k\ge1$ and $\ell\ge1$ and often omit the $p$-localization symbol like $G=G_{(p)}$.

As is well-known, any subgroup of an abelian group is normal.
But the analogue of this fact for a homomorphism between infinite loop spaces does not hold as in the following example.
\begin{ex}
Let $f\colon H\to G$ be a homomorphism such that the delooping is homotopic to the map between the Eilenberg--MacLane spaces $g\colon K(\mathbb{Q},2n)\to K(\mathbb{Q},4n)$ classified by the square $u^2\in H^{4n}(K(\mathbb{Q},2n);\mathbb{Q})$ of the generator $u\in H^{2n}(K(\mathbb{Q},2n);\mathbb{Q})$.
The homotopy fiber of $g$ is homotopy equivalent to the rationalized $2n$-sphere $S^{2n}_{(0)}$, which is not an $H$-space.
Then, since we have $\cat S^{2n}_{(0)}=1$, $f$ is not an $N_1(1)$-map by Theorem \ref{thm_quotient_Hstr}.
Similarly, the inclusion $\SO(2n)_{(0)}\to\SO(2n+1)_{(0)}$ is not an $N_1(1)$-map for $n\ge1$.
\end{ex}
But we have a normality result for some class of homomorphisms as follows.
\begin{thm}
\label{thm_normal_SU(2)}
Let $f\colon H\to G$ be a homomorphism between compact connected semisimple Lie groups. 
Let $2m-1$ and $2n-1$ denote the largest degrees of generators of the rational cohomology algebra $H^\ast(H;\mathbb{Q})$ and $H^\ast(G;\mathbb{Q})$, respectively.
Suppose that $f^\ast \colon H^\ast(G;\mathbb{Q})\to H^\ast(H;\mathbb{Q})$ is surjective and $p\ge kn+\ell m$.
Then the homomorphism $f$ is $p$-locally an $N_k(\ell)$-map.
\end{thm}
\begin{proof}
We verify the conditions in Theorem \ref{thm_main} for the trivial fiberwise topological group $E=B_kG\times H$.
We employ the technique to reduce the projective spaces in \cite[Sections 3 and 4]{MR3852292}.
Since we assume $p\ge kn+\ell m$, we have the $A_k$-equivalences
\[
H\simeq S^{2m_1-1}\times\cdots\times S^{2m_s-1}
\quad\text{and}\quad
G\simeq S^{2n_1-1}\times\cdots\times S^{2n_r-1}
\]
for some sequences $2\le m_1\le\cdots\le m_s\le m$ and $2\le n_1\le\cdots\le n_r\le n$.
Then we have the homotopy equivalences
\[
G\simeq H\times G/H
\quad\text{and}\quad
G/H\simeq S^{2q_1-1}\times\cdots\times S^{2q_{r-s}-1}
\]
for some sequence $2\le q_1\le\cdots\le q_{r-s}\le n$ since $f^\ast \colon H^\ast(G;\mathbb{Q})\to H^\ast(H;\mathbb{Q})$ is surjective.
Note that the sequences $m_1,\ldots,m_s$ and $q_1,\ldots,q_{r-s}$ are subsequences of the sequence $n_1,\ldots,n_r$.
We can consider these equivalences are $A_k$-equivalences with an appropriate $A_k$-form on $G/H$.
Consider the space
\[
B_k(i_1,\ldots,i_t)=
\bigcup_{k_1+\cdots+k_t=k}B_{k_1}S^{2i_1-1}\times\cdots\times B_{k_t}S^{2i_t-1}.
\]
Then we have the homotopy commutative diagram
\begin{align*}
\xymatrix{
B_kH \ar@{=}[r] \ar[d]_-{B_kf}
& B_kH \ar[d]^-{\text{inclusion}} \ar[r]^-\alpha
& B_k(m_1,\ldots,m_s) \ar[d]^-{\text{inclusion}}\\
B_kG \ar[r]
& \displaystyle{\bigcup_{k_1+k_2=k}B_{k_1}H\times B_{k_2}(G/H)} \ar[r]_-{\alpha\times\beta}
& B_k(n_1,\ldots,n_r)
}
\end{align*}
for maps
\[
\alpha\colon B_kH\to B_k(m_1,\ldots,m_s)
\quad\text{and}\quad
\beta\colon B_k(G/H)\to B_k(q_1,\ldots,q_{r-s})
\]
inducing injective homomorphisms of mod $p$-cohomologies, where the homotopy commutativity of the left square follows from the construction of the left bottom arrow in \cite[Proposition 3.2]{MR3852292}.
Let $\gamma\colon B_kG\to B_k(n_1,\ldots,n_r)$ denote the composite of the bottom arrows.
We also have the maps
\[
\alpha'\colon B_k(m_1,\ldots,m_s)\to BH
\quad\text{and}\quad
\gamma'\colon B_k(n_1,\ldots,n_r)\to BG
\]
such that the compositions $\alpha'\circ\alpha$ and $\gamma'\circ\gamma$ are homotopic to the inclusions.
Note that the same argument works for the $\ell$-th projective spaces.
Consider the fiberwise pointed spaces
\begin{align*}
&B_k(m_1,\ldots,m_s)
	\to
	B_k(m_1,\ldots,m_s)\times B_\ell(m_1,\ldots,m_s)
	\to
	B_k(m_1,\ldots,m_s)
	\quad\text{and}\\
&B_k(n_1,\ldots,n_r)
\to
B_k(n_1,\ldots,n_r)\times B_\ell(m_1,\ldots,m_s)
\to
B_k(n_1,\ldots,n_r)
\end{align*}
with the first projections and the first inclusions.
Thus we have the homotopy commutative diagram of fiberwise pointed spaces
\begin{align}
\label{diagram_proj_to_reduced}
\xymatrix{
B_kH\times B_\ell H \ar[r]^-{\alpha\times\alpha} \ar[d]_-{B_kf\times\operatorname{id}}
& B_k(m_1,\ldots,m_s)\times B_\ell(m_1,\ldots,m_s) \ar[d]^{\text{inclusion}} \\
B_kG\times B_\ell H \ar[r]_-{\gamma\times\alpha}
& B_k(n_1,\ldots,n_s)\times B_\ell(m_1,\ldots,m_s)
}
\end{align}
Since we suppose $p\ge kn+\ell m$, we have
\[
\pi_{2i-1}(BH)=\pi_{2i-1}(BG)=0
\]
for $i\le 2kn+2\ell m$.
This implies that there exists a commutative diagram of fiberwise pointed maps
\begin{align}
\label{diagram_reduced_to_classifying}
\xymatrix{
B_k(m_1,\ldots,m_s)\times B_\ell(m_1,\ldots,m_s) \ar[r] \ar[d]_-{\text{inclusion}}
& EH\times_HBH \ar[d]^-{\text{induced map of $f$}} \\
B_k(n_1,\ldots,n_s)\times B_\ell(m_1,\ldots,m_s) \ar[r]
& EG\times_GBG
}
\end{align}
Here the top and bottom arrows cover $\alpha'$ and $\gamma'$, respectively.
Moreover, their restrictions to the fiber over the basepoint are $\alpha'$ and the composite of the maps
\[
B_\ell(m_1,\ldots,m_s)
\xrightarrow{\text{inclusion}}
B_\ell(n_1,\ldots,n_s)
\xrightarrow{\gamma'}
BG,
\]
respectively.
Composing the squares (\ref{diagram_proj_to_reduced}) and (\ref{diagram_reduced_to_classifying}), we get the homotopy commutative square of fiberwise pointed maps
\[
\xymatrix{
B_kH\times B_\ell H \ar[rr]^-{\text{inclusion}} \ar[d]_-{B_kf\times\operatorname{id}}
& & EH\times_HBH \ar[d]^-{\text{induced map of $f$}} \\
B_kG\times B_\ell H \ar[rr]
& & EG\times_GBG
}
\]
where the bottom arrow covers the inclusion $B_kG\to BG$.
Thus taking the adjoint to this diagram in the sense of Theorem \ref{thm_Tsu16_fw}, we obtain the homotopy commutative diagram of fiberwise $A_\ell$-maps between fiberwise topological groups
\[
\xymatrix{
E_kH\times_{\conj}H \ar[rr]^-{\text{inclusion}} \ar[d]_-{\phi'}
& & EH\times_{\conj}H \ar[d]^-{\text{induced map of $f$}} \\
E \ar[rr]_-\psi
& & EG\times_{\conj}G
}
\]
where each arrow covers the same map as the corresponding arrow in the previous diagram.
The map $\phi'$ can be inverted over $B_kG$ to be a fiberwise $A_\ell$-equivalence $\phi\colon(B_kf)^\ast E\to E_kH\times_{\conj}H$.
Now we can verify the conditions in Theorem \ref{thm_main}.
Therefore, the homomorphism $f$ is $p$-locally an $N_k(\ell)$-map.
\end{proof}

\begin{rem}
In the proof of the theorem, the action $G\to\mathcal{A}_\ell(H,H)$ of the resulting $N_k(\ell)$-map is null-homotopic since $E$ is trivial.
Let us propose to call such an $N_k(\ell)$-map a \textit{$Z_k(\ell)$-map}, where $Z$ comes from the German word \textit{zentral}.
Although it seems stronger condition than being an $N_k(\ell)$-map, we have no example at this point for an $N_k(\ell)$-map but not a $Z_k(\ell)$-map between Lie groups.
\end{rem}

Let us investigate the non-normality of the inclusion $\SU(m)\to\SU(n)$.
\begin{thm}
\label{thm_non-normal_SU}
Suppose $\max\{kn-m,(k-1)n+2\}<p\le kn+(\ell-1)m$ for some $2\le m<n$ and $k,\ell\ge1$, then the inclusion $\SU(m)\to\SU(n)$ is not $p$-locally an $N_k(\ell)$-map.
\end{thm}

Before proving the theorem, we recall the cohomology of projective spaces.
\begin{lem}
\label{lem_cohproj}
If the $p$-local cohomology of a compact connected Lie group $G$ is given by
\[
H^\ast(G;\mathbb{Z}_{(p)})
=\Lambda_{\mathbb{Z}_{(p)}}(x_1,\ldots,x_n),
\]
then
\[
H^\ast(B_kG;\mathbb{Z}_{(p)})
=
\mathbb{Z}_{(p)}[y_1,\ldots,y_n]/(y_1,\ldots,y_n)^{k+1}\oplus S_k,
\]
where $y_i$ is the image of the transgression of $x_i$ in $H^\ast(BG;\mathbb{Z}_{(p)})$ and $S_k$ is a free $\mathbb{Z}_{(p)}$-module and mapped to $0$ in $H^\ast(B_{k-1}G;\mathbb{Z}_{(p)})$. 
Moreover, when the $p$-local cohomology of a compact connected Lie group $H$ is given by
\[
H^\ast(H;\mathbb{Z}_{(p)})
=\Lambda_{\mathbb{Z}_{(p)}}(x'_1,\ldots,x'_m)
\]
and a homomorphism $f\colon H\to G$ induces a surjective map on $p$-local cohomology, then the induced map
\[
(B_kf)^\ast\colon H^\ast(B_kG;\mathbb{Z}_{(p)})\to H^\ast(B_kH;\mathbb{Z}_{(p)})
\]
is also surjective.
\end{lem}
\begin{proof}
As in \cite{MR1309142,MR760190}, consider the spectral sequence converging to $H^\ast(BG;\mathbb{Z}_{(p)})$ induced from the filtration
\[
\ast=B_0G\subset B_1G\subset\cdots\subset B_kG\subset\cdots\subset BG.
\]
The differential in the $E_1$-page coincides with the differential of the cobar construction of the exterior algebra $H^\ast(G;\mathbb{Z}_{(p)})$ and then, as is well-known, the $E_2$-page is the polynomial algebra $\mathbb{Z}_{(p)}[y_1,\ldots,y_n]$.
This implies that the spectral sequence collapses at the $E_2$-page.
Truncate this spectral sequence to compute $H^\ast(B_kG;\mathbb{Z}_{(p)})$, of which the $E_1$-page looks like
\[
0\to E_1^{0,\ast}\xrightarrow{d_1}E_1^{1,\ast}\xrightarrow{d_1}\cdots\xrightarrow{d_1}E_1^{k,\ast}\to0.
\]
When $i<k$, the $E_2^{i,\ast}$-term of this spectral sequence is a free module of which the basis is given by monomials of length $i$.
Comparing with the spectral sequence for $H^\ast(BG;\mathbb{Z}_{(p)})$, the $E_2^{k,\ast}$-term is the direct sum of a similar module generated by monomials of length $k$ and the submodule $S_k$ which is isomorphic to a submodule in the $E_1^{k+1,\ast}$-term of the spectral sequence for $H^\ast(BG;\mathbb{Z}_{(p)})$.
This implies that $S_k$ is free.
Again comparing with the spectral sequence for $H^\ast(BG;\mathbb{Z}_{(p)})$, the one for $H^\ast(B_kG;\mathbb{Z}_{(p)})$ also collapses at the $E_2$-page.
Then we obtain $H^\ast(B_kG;\mathbb{Z}_{(p)})$ as in the lemma.
For the latter half, note that the homomorphism $f^\ast\colon H^\ast(G;\mathbb{Z}_{(p)})\to H^\ast(H;\mathbb{Z}_{(p)})$ is a map of coalgebra admitting a section.
Thus the induced map on the $E_2$-page is surjective, completing the proof.
\end{proof}

The $p$-local cohomology of $B\SU(n)$ is given by 
\[
H^\ast(B\SU(n);\mathbb{Z}_{(p)})
=
\mathbb{Z}_{(p)}[c_2,\ldots,c_n],
\]
where $c_i$ denotes the $i$-th Chern class of degree $2i$.
We assume that the inclusion $\SU(m)\to\SU(n)$ is $p$-locally an $N_k(\ell)$-map and derive a contradiction.
As in Theorem \ref{thm_Nkl_fwproj}, we have a fiberwise projective space $\mathcal{E}=\mathscr{B}_\ell E$ for some fiberwise $A_\ell$-space $E$ over $B_k\SU(n)$ with fiber $A_\ell$-equivalent to $\SU(m)$ and fiberwise pointed maps $\Phi\colon E_k\SU(m)\times_{\SU(m)}B_\ell\SU(m)\to\mathcal{E}_{B_k\SU(m)}$ and $\Psi\colon\mathcal{E}\to E_k\SU(n)\times_{\SU(n)}B\SU(m)$ satisfying the properties in Theorem \ref{thm_Nkl_fwproj}.

\begin{lem}
\label{lem_submodule_H(B_lE)}
The $p$-local cohomology of $\mathcal{E}$ is given by
\begin{align*}
&H^\ast(\mathcal{E};\mathbb{Z}_{(p)})\\
&=(\mathbb{Z}_{(p)}[c_2^B,\ldots,c_n^B]/(c_2^B,\ldots,c_n^B)^{k+1}\oplus S_k^B)
	\otimes(\mathbb{Z}_{(p)}\{(c_2^F)^{i_2}\cdots(c_m^F)^{i_m}\mid 0\le i_2+\cdots+i_m\le\ell\}\oplus S_\ell^F)
\end{align*}
as a $H^\ast(B_k\SU(n);\mathbb{Z}_{(p)})=\mathbb{Z}_{(p)}[c_2^B,\ldots,c_n^B]/(c_2^B,\ldots,c_n^B)^{k+1}\oplus S_k^B$-module, where we write $c_j^B=\Psi^\ast(c_j\times1)$, $c_j^F=\Psi^\ast(1\times c_j)$, $S_k^B$ and $S_\ell^F$ are free modules and $\mathbb{Z}_{(p)}S$ with a set $S$ denotes the free $\mathbb{Z}_{(p)}$-module with basis $S$.
\end{lem}

\begin{proof}
Since the rationalization $\SU(n)_{(0)}$ is $A_\infty$-equivalent to a topological abelian group, the rational cohomology Serre spectral sequence of $E_k\SU(n)\times_{\SU(n)}B_\ell\SU(n)\to B_k\SU(n)$ collapses at the $E_2$-page.
Then the $p$-local cohomology spectral sequence $\tilde{E}_\ast$ also collapses since its $E_2$-page is free by Lemma \ref{lem_cohproj}.
Consider the $p$-local cohomology Serre spectral sequence $E_\ast$ of $\mathcal{E}\to B_k\SU(n)$.
The $E_2$-term is given as
\[
E_2
=
(\mathbb{Z}_{(p)}[c_2^B,\ldots,c_n^B]/(c_2^B,\ldots,c_n^B)^{k+1}\oplus S_k^B)
\otimes
(\mathbb{Z}_{(p)}[c_2^F,\ldots,c_\ell^F]/(c_2^F,\ldots,c_m^F)^{\ell+1}\oplus S_\ell^F)
\]
for some free modules $S_k^B$ and $S_\ell^F$.
Again by Lemma \ref{lem_cohproj}, the induced map $\Psi^\ast\colon \tilde{E}_2\to E_2$ is surjective.
This implies that $E_\ast$ also collapses.
Thus the lemma follows.
\end{proof}

By the inequalities
\begin{align*}
n+p-1
&\ge
n+((k-1)n+3)-1
=kn+2,\\
(m+1)+p-1
&\le
m+1+(kn+(\ell-1)m)-1
=kn+\ell m,
\end{align*}
we can find integers $0\le\ell'\le\ell$, $m+1\le i\le n$ and $2\le j<m$ or $j=0$ satisfying
\[
i+p-1=kn+\ell'm+j
\quad\text{and}\quad
\ell'm+j\le\ell m.
\]
By the assumption $kn-m<p$ and this equality, one can see the inequalities
\begin{align*}
k<\frac{p}{3}+1
\quad\text{and}\quad
\ell'\le\frac{p}{2}.
\end{align*}
From now on we shall work in the mod $p$ cohomology, where we use the same symbols as in the $p$-local cohomology.
Since the map $\Psi\colon\mathcal{E}\to E\SU(n)\times_{\SU(n)}B\SU(n)\cong B\SU(n)\times B\SU(n)$ is fiberwise pointed and the composite $\Psi\circ\Phi\colon E\SU(m)\times_{\SU(m)}B\SU(m)\to E\SU(n)\times_{\SU(n)}B\SU(n)$ is homotopic to the induced map of the homomorphism $f\colon\SU(m)\to\SU(n)$, we have
\begin{align}
\label{align_F(1xc)}
\Psi^\ast(1\times c_i)=c_i^B+(\text{a polynomial}),
\end{align}
where the polynomial consists of terms divisible by at least one of $c_{m+1}^B,\ldots,c_{i-1}^B$.
Let us compute $\mathcal{P}^1(\Psi^\ast(1\times c_i))$ the image of Steenrod operation $\mathcal{P}^1$.

\begin{lem}
The coefficient of $c_jc_m^{\ell'}c_n^k$ ($c_0=1$ when $j=0$) in $\mathcal{P}^1c_i\in H^{2i+2p-2}(B\SU(n);\mathbb{F}_p)$ is nonzero.
\end{lem}
\begin{proof}
Recall the mod $p$ Wu formula \cite{MR454974}
\begin{align*}
\mathcal{P}^1c_i
=\sum_{2i_2+\cdots+ni_n=i+p-1}&(-1)^{i_2+\cdots+i_n-1}
\frac{(i_2+\cdots+i_n-1)!}{i_2!\cdots i_n!}\\
&\cdot\left(i+p-1-\frac{\sum_{j=2}^{i-1}(i+p-1-j)i_j}{i_2+\cdots+i_n-1}\right)
c_2^{i_2}\cdots c_n^{i_n}.\notag
\end{align*}
Let $\ell''=\ell'$ if $j=0$ and $\ell''=\ell'+1$ otherwise.
Since we have
\begin{align*}
&i+p-1-\frac{\sum_{j=2}^{i-1}(i+p-1-j)i_j}{i_2+\cdots+i_n-1}\\
&=
kn+\ell'm+j-\frac{(kn+\ell'm+j)\ell''-j-\ell'm}{k+\ell''-1}\\
&=
\frac{k((k-1)n+\ell'm+j)}{k+\ell''-1},
\end{align*}
we can compute the coefficient of $c_jc_m^{\ell'}c_n^k$ as
\begin{align*}
&(-1)^{k+\ell''-1}
	\frac{(k+\ell''-1)!}{k!\ell'!}
	\cdot\frac{k((k-1)n+\ell'm+j)}{k+\ell''-1}\\
&=(-1)^{k+\ell''-1}\frac{(k+\ell''-2)!}{(k-1)!\ell'!}((k-1)n+\ell'm+j).
\end{align*}
Now we have
\begin{align*}
k+\ell''-2
<
\frac{p}{3}+1+\frac{p}{2}-2
<p
\quad
\text{and}
\quad
0<(k-1)n+\ell'm+j=i+p-1-n<p
\end{align*}
and hence the coefficient of $c_jc_m^{\ell'}c_n^k$ is nonzero.
\end{proof}

This lemma implies that the coefficient of $(c_n^B)^kc_j^F(c_m^F)^{\ell'}$ in $\mathcal{P}^1\Psi^\ast(1\times c_i)$ is nonzero.
But it cannot happen as follows.

\begin{lem}
The coefficient of $(c_n^B)^kc_j^F(c_m^F)^{\ell'}$ ($c^F_0=1$ when $j=0$) in $\mathcal{P}^1\Psi^\ast(1\times c_i)$ is zero.
\end{lem}

\begin{proof}
Applying $\mathcal{P}^1$ to the equality (\ref{align_F(1xc)}), the term $(c_n^B)^kc_j^F(c_m^F)^{\ell'}$ comes from $(\mathcal{P}^1c_s^B)c_j^F(c_m^F)^{\ell'}$ with $s\ne n$.
Here we have $m<s<n$.
Since $\mathcal{P}^1c_s^B$ must contain the term $(c_n^B)^k$, $s$ is computed as $s=kn-(p-1)$.
But this contradicts to our assumption $s>m>kn-p$ in the theorem.
Thus the coefficient of $(c_n^B)^kc_j^F(c_m^F)^{\ell'}$ is zero.
\end{proof}

This contradicts to the previous result, completing the proof of Theorem \ref{thm_non-normal_SU}.

Let us examine the $p$-local normality of the inclusion $\SU(2)\to\SU(3)$ for small primes $p$.
By Theorems \ref{thm_normal_SU(2)} and \ref{thm_non-normal_SU}, we obtain Table \ref{table}, where \cmark\hspace{0.1em} and \xmark\hspace{0.1em} mean the inclusion is an $N_k(\ell)$-map and is not an $N_k(\ell)$-map, respectively, and $?$ means we cannot determine the normality from these theorems.
\begin{table}[h]
    \begin{subtable}{0.3\textwidth}
        \centering
\begin{tabular}{c|ccccc} 
	$k$ & $1$ & $2$ & $3$ & $4$ & $5$ \\ 
	\hline
	$N_k(1)$ & \xmark & \xmark & \xmark & \xmark & \xmark \\ 
	$N_k(2)$ & \xmark & \xmark & \xmark & \xmark & \xmark \\ 
	$N_k(3)$ & \xmark & \xmark & \xmark & \xmark & \xmark \\
	$N_k(4)$ & \xmark & \xmark & \xmark & \xmark & \xmark \\
	$N_k(5)$ & \xmark & \xmark & \xmark & \xmark & \xmark
\end{tabular}
    $p=3$
    \end{subtable}
    \begin{subtable}{0.3\textwidth}
        \centering
\begin{tabular}{c|ccccc} 
	$k$ & $1$ & $2$ & $3$ & $4$ & $5$ \\ 
	\hline
	$N_k(1)$ & \cmark & ? & ? & ? & ? \\ 
	$N_k(2)$ & \xmark & \xmark & \xmark & \xmark & \xmark \\
	$N_k(3)$ & \xmark & \xmark & \xmark & \xmark & \xmark \\
	$N_k(4)$ & \xmark & \xmark & \xmark & \xmark & \xmark \\
	$N_k(5)$ & \xmark & \xmark & \xmark & \xmark & \xmark
\end{tabular}
    $p=5$
    \end{subtable}\\
\vspace{2ex}
    \begin{subtable}{0.3\textwidth}
        \centering
\begin{tabular}{c|ccccc} 
	$k$ & $1$ & $2$ & $3$ & $4$ & $5$ \\ 
	\hline
	$N_k(1)$ & \cmark & ? & ? & ? & ? \\ 
	$N_k(2)$ & \cmark & \xmark & \xmark & \xmark & \xmark \\
	$N_k(3)$ & \xmark & \xmark & \xmark & \xmark & \xmark \\
	$N_k(4)$ & \xmark & \xmark & \xmark & \xmark & \xmark \\
	$N_k(5)$ & \xmark & \xmark & \xmark & \xmark & \xmark
\end{tabular}
    $p=7$
    \end{subtable}
    \begin{subtable}{0.3\textwidth}
        \centering
\begin{tabular}{c|ccccc} 
	$k$ & $1$ & $2$ & $3$ & $4$ & $5$ \\ 
	\hline
	$N_k(1)$ & \cmark & \cmark & \cmark & ? & ? \\ 
	$N_k(2)$ & \cmark & \cmark & \xmark & \xmark & \xmark \\
	$N_k(3)$ & \cmark & ? & \xmark & \xmark & \xmark \\
	$N_k(4)$ & \cmark & \xmark & \xmark & \xmark & \xmark \\
	$N_k(5)$ & \xmark & \xmark & \xmark & \xmark & \xmark
\end{tabular}
    $p=11$
    \end{subtable}
\vspace{2ex}
\caption{The $p$-local higher homotopy normality of $\SU(2)\subset\SU(3)$}
\label{table}
\end{table}

We leave the remaining cases for now.
The first undetermined cases are as follows.
\begin{problem}
Determine whether the inclusion $\SU(2)\to\SU(3)$ is $5$-locally an $N_2(1)$-map or not.
More generally, find methods to determine when the inclusion $\SU(m)\to\SU(n)$ is $p$-locally an $N_k(1)$-map.
\end{problem}
We should note that $E_2\SU(2)\times_{\conj}\SU(2)$ is not $5$-locally fiberwise based equivalent to the trivial bundle though its restriction $E_1\SU(2)\times_{\conj}\SU(2)$ over $B_1\SU(2)$ is $5$-locally trivial.

We close this paper giving the corresponding results for $\SO(2m+1)\to\SO(2n+1)$, which is $p$-locally homotopy equivalent to $\Sp(m)\to\Sp(n)$ for odd $p$.
The proof proceeds as Theorems \ref{thm_non-normal_SU} using the mod $p$ Wu formula 
\begin{align*}
\mathcal{P}^1p_i
=\sum_{i_1+\cdots+ni_n=i+\frac{p-1}{2}}&(-1)^{i_1+\cdots+i_n-1+i+\frac{p-1}{2}}
\frac{(i_1+\cdots+i_n-1)!}{i_1!\cdots i_n!}\\
&\cdot\left(2i+p-1-\frac{\sum_{j=1}^{i-1}(2i+p-1-2j)i_j}{i_1+\cdots+i_n-1}\right)
p_1^{i_1}\cdots p_n^{i_n}
\end{align*}
for the Pontryagin classes $p_1,\ldots,p_n$.

\begin{thm}
\label{thm_non-normal_Sp}
Suppose $\max\{2kn-2m,2(k-1)n+2\}<p\le2kn+2(\ell-1)m-1$ for some $1\le m<n$ and $k,\ell\ge1$, then the inclusion $\SO(2m+1)\to\SO(2n+1)$ is not $p$-locally an $N_k(\ell)$-map.
\end{thm}

\bibliographystyle{alpha}
\bibliography{2021normality}
\end{document}